\documentclass[11pt]{article}

\usepackage[
  shownumpages,               
  bgcolor={245,245,250},      
  braincolor={190,32,240},    
  citingstyle=authoryear,     
  bibliostyle=plainnat,       
  bibfile=references          
]{brainlab}

\usepackage{amsmath,amsfonts,bm}

\def\eqref#1{equation~\ref{#1}}

\def\1{\bm{1}}

\DeclareMathAlphabet{\mathsfit}{\encodingdefault}{\sfdefault}{m}{sl}
\SetMathAlphabet{\mathsfit}{bold}{\encodingdefault}{\sfdefault}{bx}{n}

\usepackage{algorithm}
\usepackage{algpseudocode}

\newcommand{\EE}{\mathbb{E}}
\newcommand{\RR}{\mathbb{R}}

\newcommand{\FF}{\mathcal{F}}
\newcommand{\GG}{\mathcal{G}}
\newcommand{\OO}{\mathcal{O}}
\newcommand{\e}{\varepsilon}

\def\<#1,#2>{\left\langle #1,#2 \right\rangle}

\usepackage{amsmath,amsfonts,amssymb,amsthm,mathtools}

\DeclareBrainTcbTheorem{proposition}{Proposition}
\DeclareBrainTcbTheorem{lemma}{Lemma}
\DeclareBrainTcbTheorem{corollary}{Corollary}
\DeclareBrainTcbTheorem{remark}{Remark}
\DeclareBrainTcbTheorem{claim}{Claim}
\DeclareBrainTcbTheorem{example}{Example}

\usepackage{url}
\usepackage{subfig}
\usepackage{graphicx}
\usepackage[flushleft]{threeparttable}
\usepackage{pifont}
\usepackage[table]{xcolor}
\definecolor{bgcolor2}{rgb}{0.8,1,1}

\def\<#1,#2>{\left\langle #1,#2 \right\rangle}

\setbrainmeta{
  title={Adaptive Regularized Newton Method with Inexact Hessian},
  authors={
    Aleksandr Shestakov\textsuperscript{1}, Nail Bashirov\textsuperscript{1}, Andrei Semenov\textsuperscript{4}, Alexander Gasnikov\textsuperscript{3}, Martin Tak\'{a}\v{c}\textsuperscript{2}, Aleksandr Beznosikov\textsuperscript{1}, Dmitry Kamzolov\textsuperscript{2}
  },
  affiliations={
    \textsuperscript{1}Basic Research of Artificial Intelligence Laboratory (BRAIn Lab) \\
    \textsuperscript{2}Mohamed bin Zayed University of Artificial Intelligence (MBZUAI) \\
    \textsuperscript{3}Innopolis University\\
    \textsuperscript{4}\'Ecole Polytechnique F\'ed\'erale de Lausanne
  },
  abstract={
   Newton’s method is the most widespread high-order method, demanding the gradient and the Hessian of the objective function. However, one of the main disadvantages of Newton's method is its lack of global convergence and high iteration cost. Both these drawbacks are critical for modern optimization motivated primarily by current applications in machine learning. In this paper, we introduce a novel algorithm to deal with these disadvantages. Our method can be implemented with various Hessian approximations, including methods that use only the first-order information. Thus, computational costs might be drastically reduced. Also, it can be adjusted to problems' geometries via the usage of different Bregman divergences. The proposed method converges for nonconvex and convex problems globally and it has the same rates as other well-known methods that lack mentioned properties. We present experiments validating our method performs according to the theoretical bounds and shows competitive performance among other Newton-based methods.
  }
}

\begin{document}
\begin{mainpart}

\section{Introduction}
\label{sec:intro}
In this paper, we are interested in the unconstrained optimization problem
\begin{equation}
\label{eq:main}
\min_{x \in \RR^n} f(x).
\end{equation}
Some of the most commonly used methods to solve such problems are gradient-based methods \cite{polyak1963gradient, nesterov1983method, frank1956algorithm}. The main reason for this is its low iteration time and established convergence guarantees. However, sometimes rates for these first-order methods might be rather slow for solving modern optimization problems. These problems involve not only strongly convex and convex problems, but also more practical and general nonconvex ones. For instance, as it was shown in \cite{zhang2024transformers}, first-order methods perform poorly on tasks with the "heterogeneous" Hessian structure. 

This suggests that in order to speed up convergence we need to use more detailed information about the objective function. As expected, more information can increase the rates significantly. The simplest algorithm that use the second derivative of the function is Newton's method. Nevertheless, it also faces some obstacles, for instance, it does not converge globally \cite{nesterov2018lectures} and is compatible only with strongly convex problems, which is a huge drawback while working with practical applications.

There are several ways to get rid of these issues: adding a certain step size \cite{kantorovich1948newton}, applying cubic regularization \cite{nesterov2006cubic}, using adaptive quadratic regularization \cite{doikov2024gradient} and others. But all these methods as well as classical Newton's method heavily rely on the Hessian computing, which is time and resource consuming. As an answer to that quasi-second-order algorithms \cite{nocedal1999numerical} are used, where the Hessian is approximated with less resources spent. Besides, Adam \cite{kingma2014adam} can be regarded to as a Newton-like method with the Hessian approximation. Being one of the most used machine learning's optimizers, Adam proposes an idea, that high-order methods, where derivatives are approximated inexactly and cheaply, achieve great results for well demanded applications. This proposal is also proven by AdaHessian \cite{yao2021adahessian}, SOPHIA \cite{liu2024sophiascalablestochasticsecondorder} and other methods.

Moreover, exploring high-order methods for solving problems in a more general statement is an important theoretical aspect. In particular, we are interested in deriving convergence guarantees not only for convex problems, that are thoroughly studied \cite{tyrrell1970convex}, but also for nonconvex ones, that are frequently encountered in practice \cite{goodfellow2016deep}. Another way to generalize problems' statements is to change the way we measure distance between two points. Well-established approach to this is introducing the Bregman divergences \cite{liu2011total} as a substitute to the usual Euclidean distance. This proposal tends to perform well on problems with non-trivial inner structure and geometry.

\section{Our contribution}
\label{sec:our_contr}
As mentioned above, Newton's method is a powerful one, but not applicable to a wide class of functions. Various modifications of this algorithm have been proposed to deal with this drawback. For instance, in \cite{agafonov2024inexact} authors implement the method with inexact derivatives. However, they use cubic regularization, that makes the iteration cost incredibly high. This issue is addressed in \cite{doikov2024gradient}, where authors propose a proper quadratic regularization generalized by usage of the Bregman divergences. Although, this method cannot be implemented with inexact Hessians. In \cite{doikov2024spectral}, they deal with this drawback, but none of the mentioned methods have convergence guarantees for nonconvex problems. Therefore, in \cite{jiang2023universal} authors cover this case, but their algorithm lack generality, as it cannot be applied with Hessian approximations and any Bregman divergence.

As it can be noted, none of the existing works cover all forms of this generalization, hence our contribution is in closing this gap and proposing a novel algorithm \textbf{HAT} -- \textbf{H}essian \textbf{A}daptive \textbf{T}rust region method, that
\begin{enumerate}
    \item Uses proper regularization term, that doesn't increase the iteration cost,
    \item Can be applied with Hessian approximations,
    \item Is compatible with different Bregman divergences for arbitrary problem structures, 
    \item Is applicable to nonconvex problems, as most of modern problems are so.
    \end{enumerate}
    Comparison with the existing methods can be seen in the Table \ref{table:method_comparison}. When \cite{agafonov2024inexact, doikov2024gradient} provide the best known rates for the convex problems and \cite{doikov2024spectral} for the nonconvex ones, \cite{jiang2023universal} and our algorithm achieve these rates for both nonconvex and convex settings.
    
    Another contribution is introducing the concept of \textit{relative inexactness}, that depends simultaneously on the Hessian approximation quality in terms of the gradient's norm. We investigate more general bound on second-order approximations: $\|\nabla^2 f(x_k) - H_k\| \leq M\|\nabla f(x_k)\|^{\beta}$ instead of $\|\nabla^2 f(x_k) - H_k\| \leq \delta$ \cite{agafonov2024inexact, doikov2024spectral}. This allows us to analyze more Hessian estimators, that increases the generality of the proposed algorithm. We experimentally show, that this bound takes place for most well-used approximations.

\begin{table*}[]
    \begin{center}
    \begin{tabular}{|c|c|c|c|c|c|c|}
    \hline
    \textbf{Reference} & \textbf{Best Rates}&\textbf{Reg} & \textbf{Hk} & \textbf{BD} & \textbf{nCVX} &\textbf{CVX}
    
    \\\hline
    \cite{agafonov2024inexact} & \ding{51} &  \ding{55}  & \ding{51} & \ding{55} & \ding{55} &\ding{51}
    \\\hline
    \cite{doikov2024gradient} &\ding{51} & \ding{51} &\ding{55} &\ding{51} &\ding{55} & \ding{51}
    \\\hline
    \cite{doikov2024spectral} &\ding{51}& \ding{51}& \ding{51}& \ding{55}& \ding{51} &\ding{55}
    \\\hline
    \cite{jiang2023universal} &\ding{51}& \ding{51}& \ding{55}& \ding{55}& \ding{51} &\ding{51}
    \\\hline
    \cellcolor{green!25}{This paper} &\cellcolor{green!25}{\ding{51}}& \cellcolor{green!25}{\ding{51}}& \cellcolor{green!25}{\ding{51}}& \cellcolor{green!25}{\ding{51}}& \cellcolor{green!25}{\ding{51}} & \cellcolor{green!25}{\ding{51}}
    \\\hline
    \end{tabular}
    \caption{Comparison of the existing state-of-the-art results for second-order methods. Notation: Best Rates stands for the best known convergence rates for the appropriate method classes, Reg stands for quadratic regularization, Hk stands for applicablity with Hessian approximations, BD stands for compatibility with Bregman divergences, nCVX stands for convergence guarantees for nonconvex problems and CVX stands for convex ones.}  
    \label{table:method_comparison}
    \end{center}
\end{table*}

\section{Family of second-order methods}\label{sec:family_methods}
In the introduction above, we mentioned the most similar works to ours, where convergence guaranteed are derived for Newton-based methods. Although, many of widely used methods with Hessian approximations, that rely on various heuristics. Algorithms with these estimations do not always have proper theoretical convergence bounds. Nevertheless, derivatives' approximations in these methods tend to approximate the exact Hessian quite well, therefore, \textbf{HAT} may be implemented with these estimations. Below we mention as established Hessian approximations as well as novel ones, that show great performance at modern application problems. 

\subsection{Quasi-Newton methods}\label{sec:quasi_newton}

\textbf{DFP} \cite{Davidon1959VariableMM,Fletcher1988PracticalMO}\textbf{.} One of the first heuristics for the true Hessian approximation is acquired from the quadratic model of objective function. Instead of $\nabla^2 f(x_{k+1})$ the matrix $B_{k+1}$ is used, which is obtained by matching the gradient of this model and the real function we gain the secant equation: $B_{k+1}(x_{k+1} - x_k) = \nabla f(x_{k+1}) - \nabla f(x_{k})$. As the solution might not be unique, we choose the least changing matrix from the previous iteration in Frobenius norm.

\textbf{BFGS} \cite{nocedal1999numerical}\textbf{.} The DFP updating formula is effective, but it is superseded by the BFGS formula, that operates with inverse approximate Hessian $H_{k} = B_{k}^{-1}$. It is obtained as the previous one via the secant equation $H_{k+1}(\nabla f(x_{k+1}) - \nabla f(x_k)) = x_{k+1} - x_k$. To derive a unique solution we use the same intuition as above. It should be noted, that we can derive $B_{k+1}$ from $B_k$ from the Sherman-Morrison-Woodburry formula, hence it's not resource-consuming.

\textbf{SR-1} \cite{Conn1991ConvergenceOQ}\textbf{.} In DFP and BFGS, new matrix differs from the previous by a 2-rank one. We can simplify this, by using the Symmetric-Rank-1 update. As usual, we use a secant equation. SR-1 update does not guarantee a positive definiteness result, therefore it might perform better for some nonconvex problems.

\textbf{Limited memory methods.} Dealing with high dimensions problems, keeping the whole "Hessian" might consume a lot of memory. One of the possible solutions is to store low dimension representations of the matrix, that allow us to reconstruct the approximation. For instance, L-BFGS \cite{Liu1989OnTL} method is the limited memory edition of the BFGS algorithm. It reconstructs the BFGS matrix with the finite number of vectors. This method is competitive to others and is used in many optimizers.

\subsection{Scaling methods}\label{sec:scaling_methods}

However, as the Quasi-Newton methods are based on a secant equation, they imitate the Hessian action on a gradient too good, hence models end up in a local minimum too soon \cite{dauphin2014identifying}. And as neural networks have enormous number of local minimums due to its nonconvex structure, these methods end up far from the optimal point. All of the above suggests that we should not use such tight Hessian approximations, but rather more general ones

\textbf{Fisher} \cite{Amari1998NaturalGW,Amari2021InformationG}\textbf{.} Another way to approximate the Hessian matrix comes from the probabilistic nature of data.  If each input $a$ has probability of being mapped into label $b$, the goal is to minimize the distance between the target joined distribution $p_x(a,b)$ and a joined distribution from the model $\widehat{p}(a,b)$. Nevertheless, computing Hessian is still time and resource consuming. Statistics tells us, that if these distributions matches, we can prove that the Hessian matrix is equivalent to the Fisher matrix $F = \EE\left[-\nabla_x^2 \log p_x(a,b)\right]$, which is equal to $\EE\left[\nabla_x \log p_x(a,b)\nabla_x \log p_x(a,b)^T\right]$. As the last matrix is significantly easier to compute, we can use it when the model has high accuracy.

\textbf{Diagonal Hessian.} Experiments show, that for MLPs and transformers Hessians' structure is not homogeneous. 
In fact, they consist of distinct blocks \cite{singh2020woodfisherefficientsecondorderapproximation,zhang2024adamminiusefewerlearning}, that impact the convergence process differently. Therefore, knowing the neural network's architecture and problem statement we can predict the certain derivatives that have the biggest impact on the convergence and use not the whole hessian, but only them instead. If these blocks are located on the matrix's main diagonal or they do not intersect we might use approximations for them. Thus, we might abuse the block structure of the Hessian to perform significantly less computations.

\textbf{Adagrad} \cite{duchi2011adaptive}\textbf{.} Based on a diagonal representation of the Hessian, the Adagrad method does not use any second order information about the function, but use gradients' outer product to imitate the Hessian. To prevent aggressive, monotonically decreasing rate, the momentum can be added, hence we keep the information about the previous iterations as  in \textbf{RMSProp} \cite{tieleman2017divide}.These momentum-like methods tend to perform well on practical applications, as they can avoid local minimums and do not demand high complexity cost. Similar approach can be found in \textbf{Adam} \cite{kingma2014adam}.

\textbf{OASIS} \cite{jahani2021doublyadaptivescaledalgorithm}\textbf{.} Taking into account not only the diagonal structure of the Hessian and momentum, but also the changing loss surface curvature across different dimensions we might get the OASIS-like algorithms. Experiments show, that truncating preconditioned matrices leads to a better search direction and adaptive step sizes help to achieve better convergence rates. This approach is continued by the method \textbf{SOPHIA} \cite{liu2024sophiascalablestochasticsecondorder}, that performs better for transformers.

\textbf{K-FAC} \cite{martens2015optimizing}\textbf{.} The use of the Fisher information matrix as the Hessian approximation appears to be beneficial in the neural network training process. However, computational costs may not be worth the benefit gained. In answer to this, a special preconditioner for neural networks is proposed. To approximate the Fisher information matrix, which is neither low-rank nor diagonal nor block-diagonal, the K-FAC preconditioner is used. This approximation takes into account the layer structure of the neural network and uses the Kronecker product to combine layers into blocks.

\textbf{Shampoo} \cite{gupta2018shampoo}\textbf{.} When K-FAC is mostly used for generative models, another preconditioning technique is introduced, that is applicable for stochastic optimization over tensor spaces. With quite simple recursive update, this method performs well on most problems. This approach was further investigated in \textbf{SOAP} \cite{vyas2024soap}, where the authors added a regularization factor.

\subsection{Engineers' tricks}\label{sec:engineers_tricks}

Not using the whole Hessian might as well be caused by technical constraints. Therefore, some methods are used to reduce the number of computations significantly. Surprisingly, some of them do not perform much worse than exact computations and also have theoretical convergence bounds that are asymptotically the same as for the second-order methods.

\textbf{Lazy computations} \cite{doikov2023second}\textbf{.} Converging to the optimum, step sizes diminish and gradients and Hessians change insignificantly. That's how the idea of keeping a previously computed Hessian appeals. We call this lazy Hessian updates. Reusing the old Hessians can be beneficial not only during the last iterations of the algorithm, but also during the whole process, as studies show. The idea of lazy computations is applicable to the Hessian approximations as well, hence many various methods can be developed and investigated. This allows us to decrease the computations by many times and not lose the convergence process. We may find a trade-off between the frequency of updates and convergence properties.

\textbf{Compressed Hessians.} Dealing with distributed optimization, we need to transmit optimizers' parameters, gradients or high-order derivatives. Due either to technical constraints or to cost of communications nodes may receive not the whole Hessian, but its compressed version. In various studies, different compression techniques were examined that show convergence results comparable to exact methods but use significantly fewer bits to transmit \cite{islamov2021distributed, safaryan2021fednl}. As most compressors are characterized by the norm difference between the matrix and its compressed version, they are applicable to our method and can be combined with other methods.

\section{Adaptive regularization}\label{sec:adapt_reg}
\textbf{Notation.} We use the standard Euclidean norm for vectors: $\|x\| := \<x,x>^{1/2},$ $x\in \RR^n$, and corresponding operator norm on matrices: $\|A\| := \max_{x: \|x\|\leq 1}\|Ax\|$. The objective functional $f : \RR^n \rightarrow \RR$ is a twice differentiable function. We denote its global minimum $f_* := \inf_x f(x)$ which we assume finite and may be not unique. If it is unique, we denote $x_* := \arg\min_x f(x)$.  We also introduce the gradient of $f$ at point $x$ as $\nabla f(x)\in \RR^n$, and the Hessian matrix by $\nabla^2 f(x)\in \RR^{n \times n}.$ 

In order to formalize the convergence process of the proposed algorithm, we introduce several assumptions. 

\begin{assumption}\label{as:conv}
The function $f: \RR^n \rightarrow \RR~$ is convex, i.e.
\[f(y) \geq f(x) + \<\nabla f(x), y-x>,~~~ \forall x,y\in\RR^n.\]
\end{assumption}

We work with both convex and nonconvex functions, hence we derive theory for these two cases, and this assumption is not always satisfied. In nonconvex optimization, our goal is to find $\e$-approximate first-order point $x$, ensuring $\|\nabla f(x)\| \leq \e$. In convex optimization, we search for $\e$-approximate zero-order point $x$, ensuring $f(x) - f_* \leq \e$. Further assumptions are widely used in the optimization theory, especially among high-order ones. 

\begin{assumption}\label{as:heslip}
    The function $f: \RR^n \rightarrow \RR~$ has $L_2$-Lipschitz Hessian, i.e.
    \[\left\|\nabla^2 f(x) - \nabla^2 f(y)\right\| \leq L_2\|x-y\|,~~~ \forall x,y\in\RR^n\]
\end{assumption}

\begin{lemma}\label{lem:nest}\cite{nesterov2021implementable}
    Let Assumption \ref{as:heslip} be satisfied, then $\forall x,y\in\RR^n$:
    \begin{equation*}
    \begin{split}
    \left\|\nabla f(y) - \nabla f(x) - \nabla^2f(x)(y-x)\right\|\leq \frac{L_2}{2}\|y-x\|^2,
    \end{split}
    \end{equation*}
    \begin{equation*}
    \begin{split}
    &\Big|f(y) - f(x) - \<\nabla f(x),y-x>-\frac{1}{2}\<\nabla^2 f(x)(y-x),y-x>\Big|\leq \frac{L_2}{6}\|y-x\|^3.
    \end{split}
    \end{equation*}
\end{lemma}

In the convex case, we need to work with sublevel set of function $f$:
$$L:= \{x~:~f(x) < f(x_0)\}.$$

\begin{assumption}\label{as:sublevel}
    The function $f:\RR^n\rightarrow\RR$ has a finite sublevel set, i.e. 
    \[D := \sup \{\|x-x_*\|~:~ f(x) < f(x_0)\}< \infty.\]
    Since the set is bounded, the gradient's norm on this set is also bounded by some constant: $\|\nabla f(x)\| \leq G,~~~ \forall x \in L.$
\end{assumption}

Considered optimization schemes are based on some scaling function $\rho$, which we assume to satisfy several properties. Based on this function $\rho$, we define the Bregman divergence $V(x,y) := \rho(y) - \rho(x) - \<\nabla \rho(x),y-x>$, which is used to further generalize the Euclidean distance between points and to take advantage of the problem's geometry.

\begin{assumption}\label{as:rho}
    The scaling function $\rho$ is $\sigma_V$-strongly convex and has $L_V$-Lipschitz gradient, i.e.
    \[\rho(y) \geq \rho(x) + \<\nabla \rho(x),y-x> + \frac{\sigma_V}{2}\|y-x\|^2,~~~ \forall x,y \in \RR^n,\]
    \[\|\nabla \rho(y) - \nabla \rho(x)\| \leq L_V\|y-x\|, ~~~\forall x,y\in\RR^n.\]
    Moreover, we assume $2\sigma_V > L_V.$
\end{assumption}

In literature many examples of the Bregman divergences, satisfying these criteria, can be found.

\begin{example}\cite{doikov2024gradient}
Let $B$ be a positive definite symmetric matrix with eigenvalues lying in $[\lambda_{min}, \lambda_{max}]$, then with  $\rho(x) = \frac{1}{2}\<Bx,x>$ as a scaling matrix we get Bregman divergence
$$V(x,y) = \frac{1}{2}\<B(y-x),y-x>.$$
This Bregman divergence satisfy Assumption $\ref{as:rho}$ with constants $L_V = \lambda_{max}$ and $\sigma_V = \lambda_{min}$. 
\end{example}

\begin{example}\cite{doikov2024gradient}
    Solving a problem on a simplex
    
    {\small\begin{equation*}
         C_n := \left\{x \in \RR^n~:~ \sum\limits_{i=1}^n x^{(i)} = 1,~\forall~ i \in [1,n] ~~x^{(i)} \geq 0\right\},
    \end{equation*}}one of the most suitable norm is $\ell_1$-norm, defined as $\|x\|_1 := \sum_{i=1}^n |x^{(i)}|$ for $x \in \RR^n$, where $x^{(i)}$ is the $i$-th coordinate of $x$. Let $B$ be a positive definite symmetric matrix with eigenvalues lying in $[\lambda_{min}, \lambda_{max}]$  and $\theta > 0$. Then we introduce the scaling function
    \begin{equation*}
        \rho(x) = \frac{1}{2}\<Bx,x> + \theta\sum\limits_{i=1}^n(x^{(i)} + \theta)\ln(x^{(i)} + \theta).
    \end{equation*}
    Constructing the Bregman divergence from this scaling function, we satisfy Assumption \ref{as:rho} with $L_V = \lambda_{max} + 1$ and $\sigma_V = \lambda_{min} + \frac{\theta}{1+n\theta}$.
\end{example}

\subsection{Proposed algorithm}\label{sec:proposed_algo}

Here we present the motivation for our new \textbf{HAT} algorithm. As in other optimization schemes \cite{nesterov2018lectures}, we construct a model of an objective functional and proceed to minimize it. For these needs, Taylor's expansion is frequently used. Initially, a cubic regularization is introduced in \cite{nesterov2006cubic}:
\begin{equation*}
    x_{k+1} = \arg\min\limits_{y \in \RR^n} \<\nabla f(x_k), y-x_{k}> +\frac{1}{2}\<\nabla^2 f(x_{k})(y-x_{k}),y-x_{k} > + \frac{L_2}{6}\|y-x_{k}\|^3,
\end{equation*}
but its high iteration cost turns out to be unreasonably high, hence various ways to avoid it are suggested. Eventually, quadratic regularization and more general one with Bregman divergences is proposed in \cite{doikov2024gradient}:

{
\begin{equation*}
    x_{k+1} = \arg\min\limits_{y \in \RR^n} \<\nabla f(x_k), y-x_{k}> +\frac{1}{2}\<\nabla^2 f(x_{k})(y-x_{k}),y-x_{k} > +\frac{1}{\sigma}\sqrt{\frac{L_2\|\nabla f(x_k)\|}{3}}V(x_k,y).
\end{equation*}}

Although, their analysis is not applicable to the nonconvex functions, therefore the idea is to add trust region constraints to the method. In \cite{jiang2023universal} the authors implemented this, but only in the  simplest Euclidean case, where $V(x,y) = \tfrac{1}{2}\|y-x\|^2$:
\begin{equation*}
    \begin{split}
    x_{k+1} = &\arg\min\limits_{y \in \RR^n} \<\nabla f(x_k), y-x_{k}> +\frac{1}{2}\<\nabla^2 f(x_{k})(y-x_{k}),y-x_{k} > + \frac{L_2}{9}\|y-x_{k}\|^2\\
    \text{s.t.} ~ &\|y-x_k\|^2 \leq \frac{1}{9L_2} \|\nabla f(x_k)\|.
    \end{split}
\end{equation*}
We connect these methods with the opportunity to use not the true Hessian but its approximation instead and obtain the following algorithm:

\begin{algorithm}{\texttt{HAT}}\label{alg:alg1}
            \textbf{Input}: initial $x_0$, hyperparameters $\eta > 0,~ 0 < \xi < 1$
            \begin{algorithmic}
            \For {$k = 0, 1, \ldots$} 
            \State Compute Hessian approximation $H_k$
            \State Choose ($r_k, A_k$) using $\xi,\eta$
            \State Find $d_k \in \RR^n$ such that
            \begin{equation}\label{alg:subproblem}
        \begin{split}d_k = \arg\min\limits_{d \in \RR^n} ~ &\<\nabla f(x_k), d> + \frac{1}{2}\< H_kd,d > + A_k V(x_k, x_k+d)\\
        \text{s.t.} ~ &\|d\|^2 \leq r_k^2 \|\nabla f(x_k)\|
        \end{split}
\end{equation}
            \State Update $x_{k+1} = x_k + d_k$
            \EndFor
            \end{algorithmic}
\end{algorithm}

Since the subproblem (\ref{alg:subproblem}) is a quadratic minimization problem with quadratic constraints, one can apply as  projection methods, as well as special solvers to find $d_k$ in (\ref{alg:subproblem}). This choice does not have a huge impact on the convergence rates.

The main difficulty is in the choice of the parameters $r_k$ and $A_k$ -- chosen poorly, there can be no convergence at all. If $A_k$ is excessively big, the method becomes more similar to Gradient Descent, or Mirror Descent, depending on the Bregman divergence. If this parameter is too small, algorithm act as the Newton's method, that lacks global convergence. The parameter $r_k$ also influence the rate of convergence. If it is small, rates are low, as we barely move towards the optimum. When $r_k$ is too big, on the other way, we end up with an inaccurate function's model, especially for nonconvex problems or poor Hessian approximations, hence our function estimation does not look like the objective one whatsoever. Therefore, we do not converge at all.

Further we briefly explain the proof's idea. Its full version with all statements proven can be seen in Appendix \ref{sec:appendix}.

All the iterations of Algorithm \ref{alg:alg1} can be divided into two groups: ones, where constraint's inequality in (\ref{alg:subproblem}) is met and $\|d_k\|^2 = r_k^2\|\nabla f(x_k)\|,$ and others, where this inequality is strict, therefore $\|d_k\|^2 < r_k^2\|\nabla f(x_k)\|$. We investigate this two groups independently, and from bounds on both these groups' size obtain the final estimation on iterations, needed to converge to the $\e$-stationary first- or zero-order point. 

If constraints result in the strict inequality, we demand a decrease in the gradient's norm: $\|\nabla f(x_k+d_k)\| < \|\nabla f(x_k)\|$ (see Lemma \ref{lem:grad_decr}). This helps us to control the norm of the gradient, hence it converges to zero. However, sometimes nonconvex structure of the problem does not allow us to keep constraints as strict inequalities and we obtain $d_k$ on the boundary of the feasible set. In this case, we cannot control the gradient's norm properly, therefore it might drastically increase. However, we demand in this case the sufficient decrease in the function's value. Previous works \cite{doikov2024gradient, doikov2024spectral, jiang2023universal} show that this decrease should be proportional to some degree of the gradient's norm: $f(x_k+d_k) - f(x_k) \leq -\|\nabla f(x_k)\|^{3/2}$ (see Lemma \ref{lem:val_decr}).

In order to relax these strict conditions on decrease in either the gradient's norm or the function's value, hyperparameters $\eta$ and $\xi$ are introduced. The Hyperparameter $\eta > 0$ influence the rates of the convergence in the case the subproblem's (\ref{alg:subproblem}) solution is on the boundary of the feasible set:
\begin{equation}\label{eq:iter1}
    f(x_k+d_k) - f(x_k) \leq -\eta\|\nabla f(x_k)\|^{3/2}, ~~~\eta > 0.
\end{equation}
The Hyperparameter $\xi$ has an impact on the rates on convergence in the case the subproblem's (\ref{alg:subproblem}) solution is inside the feasible set:
\begin{equation}\label{eq:iter2}
    \|\nabla f(x_k+d_k)\| \leq \xi\|\nabla f(x_k)\|, ~~~\xi\in(0,1).
\end{equation}

It turns out, that we cannot perform infinitely many iterations with sufficient decrease in the value. Our next step is to bound the number of iterations with decrease in gradient's norm in a row, which is possible, since gradient's norm declines linearly. Therefore, we can perform finite number of iterations with sufficient decrease in the gradient's norm between two iterations with sufficient decrease in the value. Hence, by multiplying these bounds, we obtain a higher bound on number iterations before converging to a $\e$-stationary first-order point. In the convex case, bounds are stricter, since we have more information about the function.

Discussion above about constraint's inequality may be formalized via the Karush-Kuhn-Tucker conditions, that can be found in Appendix \ref{sec:appendix}. Finally, we can prove the following theorem:
\begin{theorem}\label{theor}
    Let $f$ and $\rho$ satisfy Assumptions \ref{as:heslip} and \ref{as:rho}. Then with the choice of $r_k$ and $A_k$ as follows:
    {\small\begin{equation*}
    \begin{split}
    r_k &= \xi\Bigg(\frac{\|\nabla ^2 f(x_k) - H_k\|}{\|\nabla f(x_k)\|^{1/2}}\left(1 + \frac{L_V}{2\sigma_V-L_V}\right)+  \sqrt{\xi\left(\frac{L_2}{2} + L_V\frac{2\eta + L_2/3}{2\sigma_V-L_V}\right)}\Bigg)^{-1},
    \end{split}
    \end{equation*}}
    {\small\begin{equation*}
    A_k = \frac{\|\nabla ^2 f(x_k) - H_k\|}{2\sigma_V-L_V} + \frac{r_k\|\nabla f(x_k)\|^{1/2}}{2\sigma_V - L_V}\left(2\eta + \frac{L_2}{3}\right),
\end{equation*}}
    every iteration of Algorithm \ref{alg:alg1} satisfies
    \begin{eqnarray*}
    f(x_k+d_k) - f(x_k) &\leq& -\eta r_k^3\|\nabla f(x_k)\|^{3/2} ~~~(*)\\ 
    \text{or}~~~\|\nabla f(x_k+d_k)\| &\leq& \xi\|\nabla f(x_k)\|~~~~~~~~~~(**).
\end{eqnarray*}
\end{theorem}
Important corollary of Theorem \ref{theor} is the monotone convergence of the proposed method: $f(x_{k+1}) \leq f(x_k)$ (see Lemma \ref{lem:monotone}), which is really advantageous for the real-world applications.

Theorem \ref{theor} allows us to obtain the needed number of iterations to obtain an $\e$-accurate solution in terms of either $\|\nabla f(x)\|$ or $f(x) - f_*$. Define $r_{min} := \inf\limits_{k}r_k$. The further intuition is the following: 

$\bullet$ In the nonconvex case, we examine all the iterations, while $\|\nabla f(x_k)\| \geq \e$. This bound combined with the monotone nature of the convergence result limits the number of iteration with the sufficient decrease in the function's value $(*)$ from above -- the method can perform no more, than $\OO\left(\tfrac{\e^{-3/2}}{\eta r_{min}^3}\right)$ of them. Furthermore, with the upper bound on the gradient's norm, one can prove, that there can be maximum $\OO\left(\log_{1/\xi}\frac{1}{\e}\right)$ iterations in a row with sufficient decrease in the gradient's norm $(**)$ between two consequent iterations with sufficient decrease in the function's value $(*)$ (see Lemma \ref{lem::set_size}). All in all, after $\OO\left(\tfrac{\e^{-3/2}\log_\xi\e}{\eta r_{min}^3}\right)$ iterations we obtain an $\e$-stationary first-order point.

$\bullet$ In the convex case, we consider all the iterations, while $f(x_k) - f_* \geq \e$. It turns out, that due to the convex structure, the proposed method can conduct no more, than $\OO\left(\tfrac{\e^{-1/2}}{\eta r_{min}}\right)$ iterations with the sufficient decrease in the function's value $(*)$ (see Lemma \ref{lem:conv_size}). Besides, with the proper choice  of hyperparameter $\xi$, one can derive another bound on the gradient's norm on these iterations $(*)$ (see Lemma \ref{lem:conv_g}). Combining gradients' bounds on iterations $(*)$ and $(**)$, we obtain, that after $\OO\left(\tfrac{\e^{-1/2}}{\eta r_{min}} + \log_\xi\e\right)$ iterations algorithm ends up in an $\e$-stationary zero-order point.

Convergence can be poor, if $r_k$ is not bounded from below. To investigate this we introduce the \textit{relative inexactness} $\Delta = \sup\limits_{k}\frac{\|\nabla ^2 f(x_k) - H_k\|}{\|\nabla f(x_k)\|^{1/2}}.$ Examining the $r_{min}$, one can notice, that $r_{min} = \frac{a}{b\Delta + c}$, where $a, b, c$ are constants. Therefore, $r_{min}$ is bounded from below if and only if $\Delta$ is bounded from above. Hence, instead of investigating $r_k$ we can inspect $\Delta$.

It is intuitive, that for decent Hessian approximations, while converging to the optimum, the scaling matrix $H_k$ does not differ too much from the exact Hessian $\nabla^2 f(x_k)$. On the contrary, if $\|\nabla ^2 f(x_k) - H_k\|$ is poorly bounded we get enormous relative inexactness while approaching the optimum. To be more precise, difference between Hessian and its approximation can be bounded in three different ways:

$\bullet$ By constant: $\|\nabla^2f(x_k) - H_k\| \leq M$. Then, with $\|\nabla f(x_k)\| \rightarrow \e$, we result in $\Delta \sim \e^{-1/2}$. To converge, it takes $\OO\left(\e^{-3}\log\frac{1}{\e}\right)$ iterations in the nonconvex case and $\OO\left(\e^{-2} + \log\frac{1}{\e}\right)$ iterations in the convex one.

$\bullet$ By insufficient gradient's degree: $\|\nabla^2 f(x_k) - H_k\| \leq M\|\nabla f(x_k)\|^{\alpha}$, where $\alpha < 1/2$. Then, with $\|\nabla f(x_k)\| \rightarrow \e$, we have $\Delta \sim \e^{\alpha-1/2}$. The method needs $\OO\left(\e^{3-3\alpha}\log\frac{1}{\e}\right)$ and $\OO\left(\e^{3\alpha-2}+\log\frac{1}{\e}\right)$ iterations in the nonconvex and convex cases, respectively.

$\bullet$ By decent gradient's degree: $\|\nabla^2 f(x_k) - H_k\| \leq M\|\nabla f(x_k)\|^{\beta}$, where $\beta \geq 1/2$. Then, with $\|\nabla f(x_k)\| \rightarrow \e$ we result in $\Delta = \OO(1)$. Therefore, we have following corollaries:

\begin{theorem}[Nonconvex case]\label{th::nonconvex}
   Let the conditions of Theorem \ref{theor} be met and $\Delta$ is bounded, then it  takes $\mathcal{O}\left(\e^{-3/2}\log \frac{1}{\e}\right)$ iterations to converge to a $\e$-approximate first-order stationary point.
\end{theorem}
\begin{theorem}[Convex case]\label{th::convex}
    Let $f$ satisfy Assumptions \ref{as:conv}, \ref{as:sublevel}, conditions of Theorem \ref{theor} be met , matrix $H_k$ is quite positive-definite and $\Delta$ is bounded, then it takes $\mathcal{O}\left(\e^{-1/2} + \log \frac{1}{\e}\right)$ iterations to converge to a $\e$-approximate zero-order stationary point.
\end{theorem}

Below we examine the relative inexactness for various Hessian approximations and discover, that it is bounded from above for all commonly used quasi-second-order methods, that makes proposed algorithm applicable for these approximations.

\section{Experiments}\label{sec:numerical_results}
To demonstrate the applicability of the suggested method, we compare it to the other classical and state-of-the-art second-order methods as well as provide evidence, that theoretical assumptions on the Hessian estimators are fulfilled.

\subsection{Relative inexactness examination}
\label{sec:relative_exactness}
This paragraph is dedicated to the experimental validation of the concepts introduced in the previous section. As we want to obtain the best possible rates, we will thoroughly examine the third case, when the norm of the difference between the Hessian and its estimator is bounded by a decent degree of the gradient's norm.

To verify our concern, we consider the dynamics of both gradient norms and the discrepancy between the true Hessian $\nabla^2 f$ and its' approximation $H$.
In particular, we report our findings for the Hutchinson's \cite{Bekas2007AnEF} preconditioner and (Empirical) Fisher approximation \cite{chaudhari2017entropysgdbiasinggradientdescent,kunstner2020limitationsempiricalfisherapproximation} (for more information see also Section \ref{sec:scaling_methods}).

\begin{figure}[H]
    \subfloat{\label{subfig:hutch_ri_roberta}%
      \includegraphics[width=0.3\linewidth]{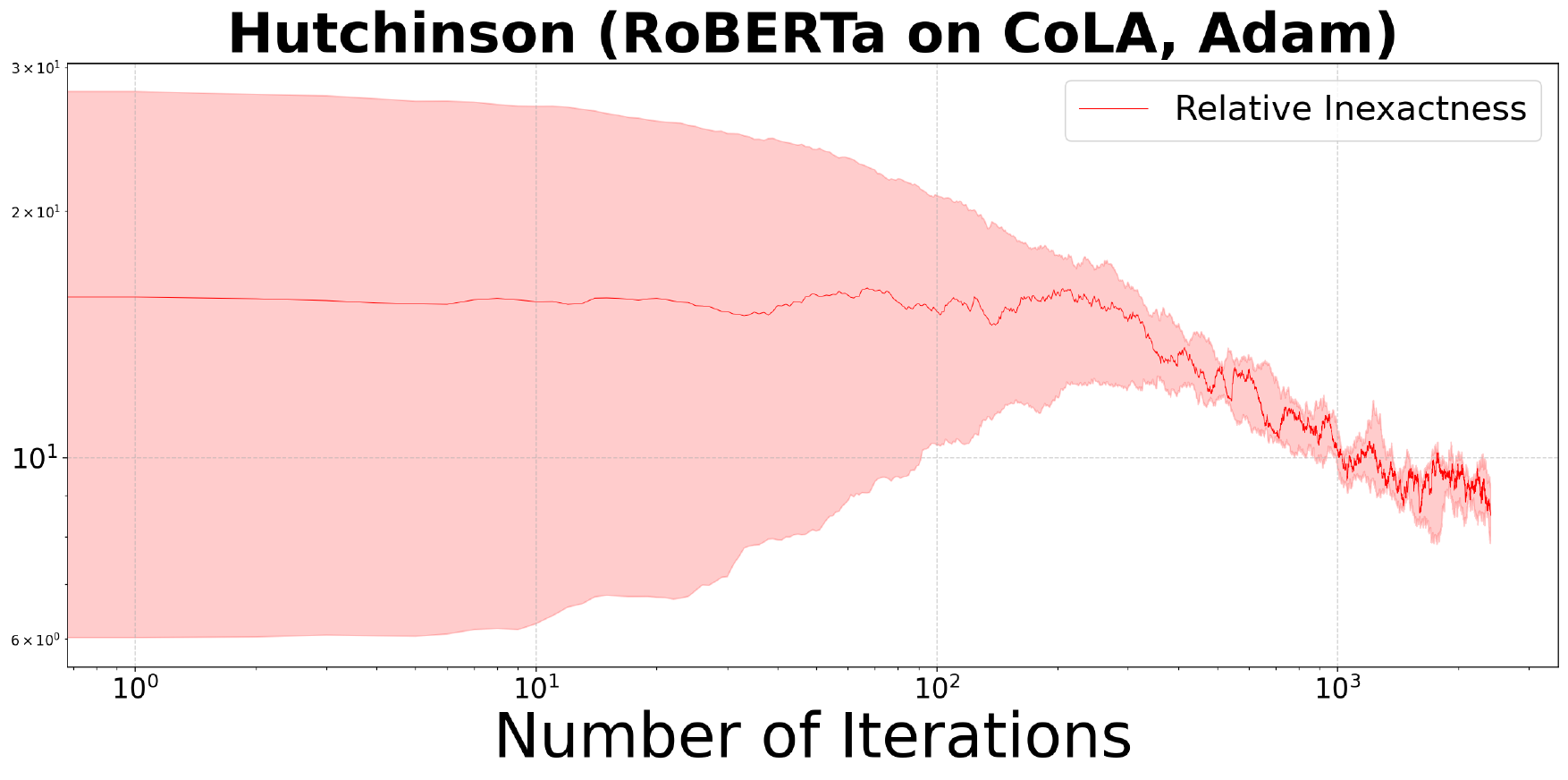}}%
    \hfill
    \subfloat{\label{subfig:fisher_ri_bert}%
      \includegraphics[width=0.3\linewidth]{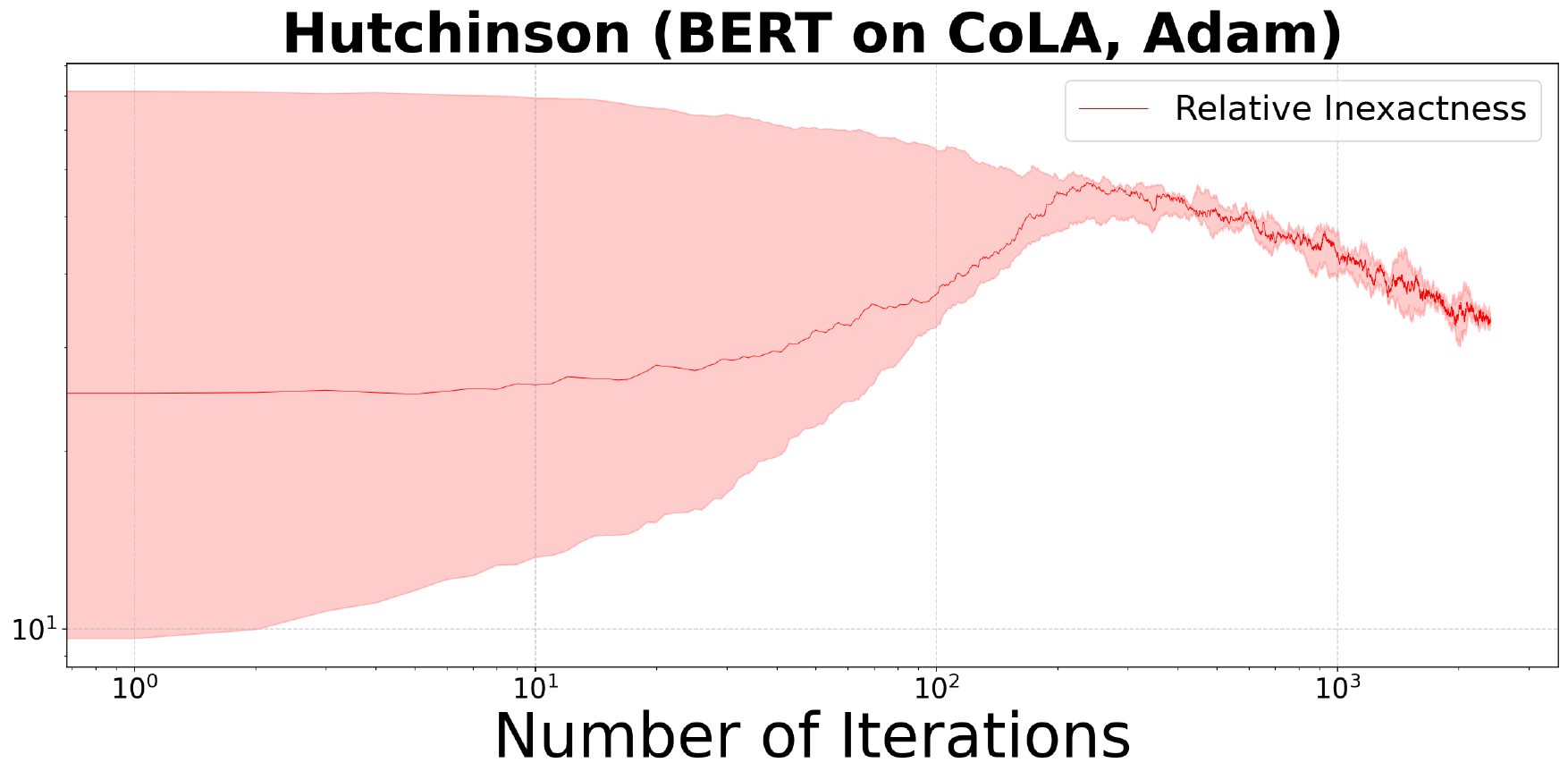}}%
    \hfill
    \subfloat{\label{subfig:hutch_ri_mlp}%
      \includegraphics[width=0.3\linewidth]{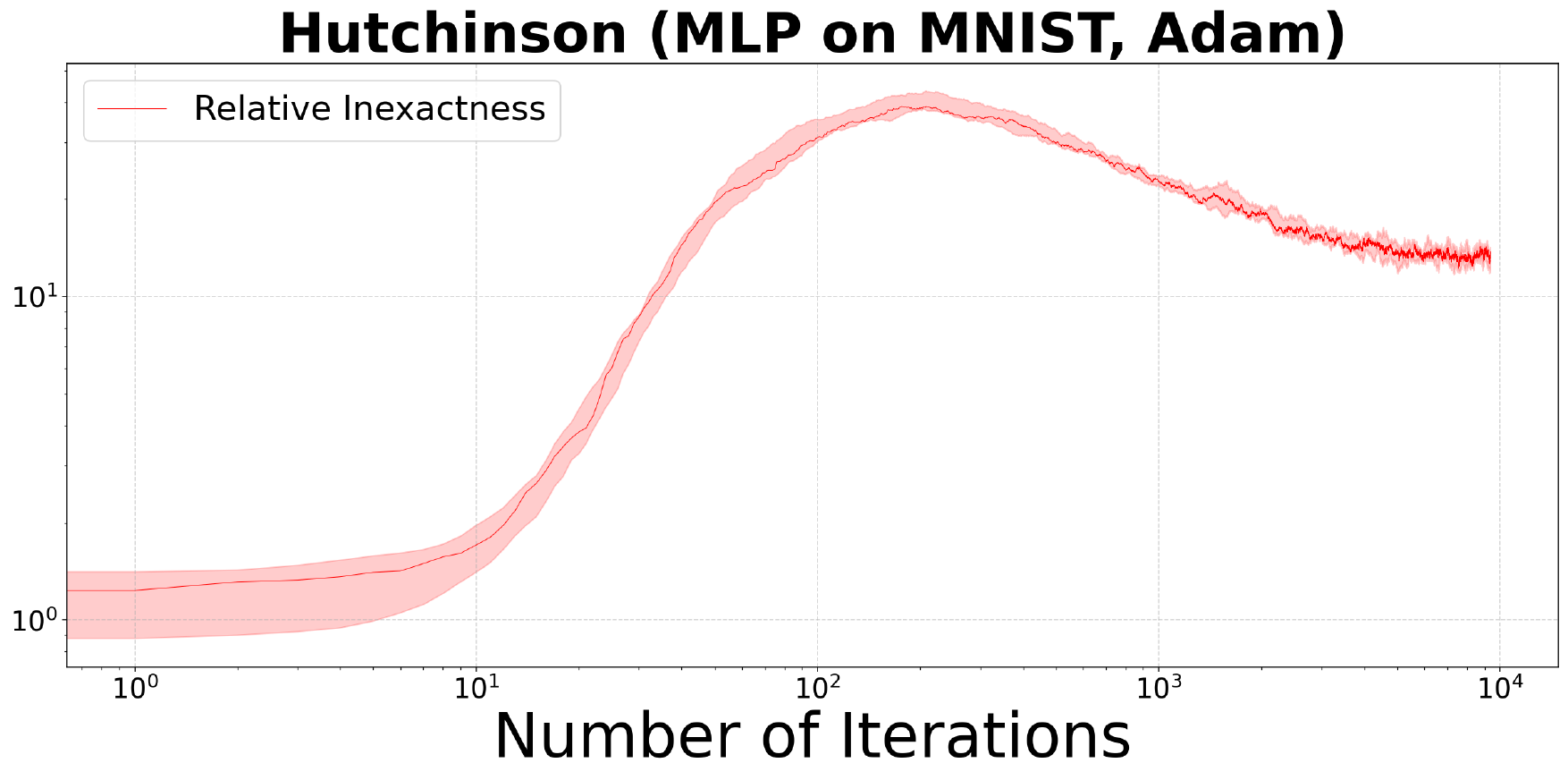}}%
    \hfill
    \subfloat{\label{subfig:fisher_ri_roberta}%
      \includegraphics[width=0.3\linewidth]{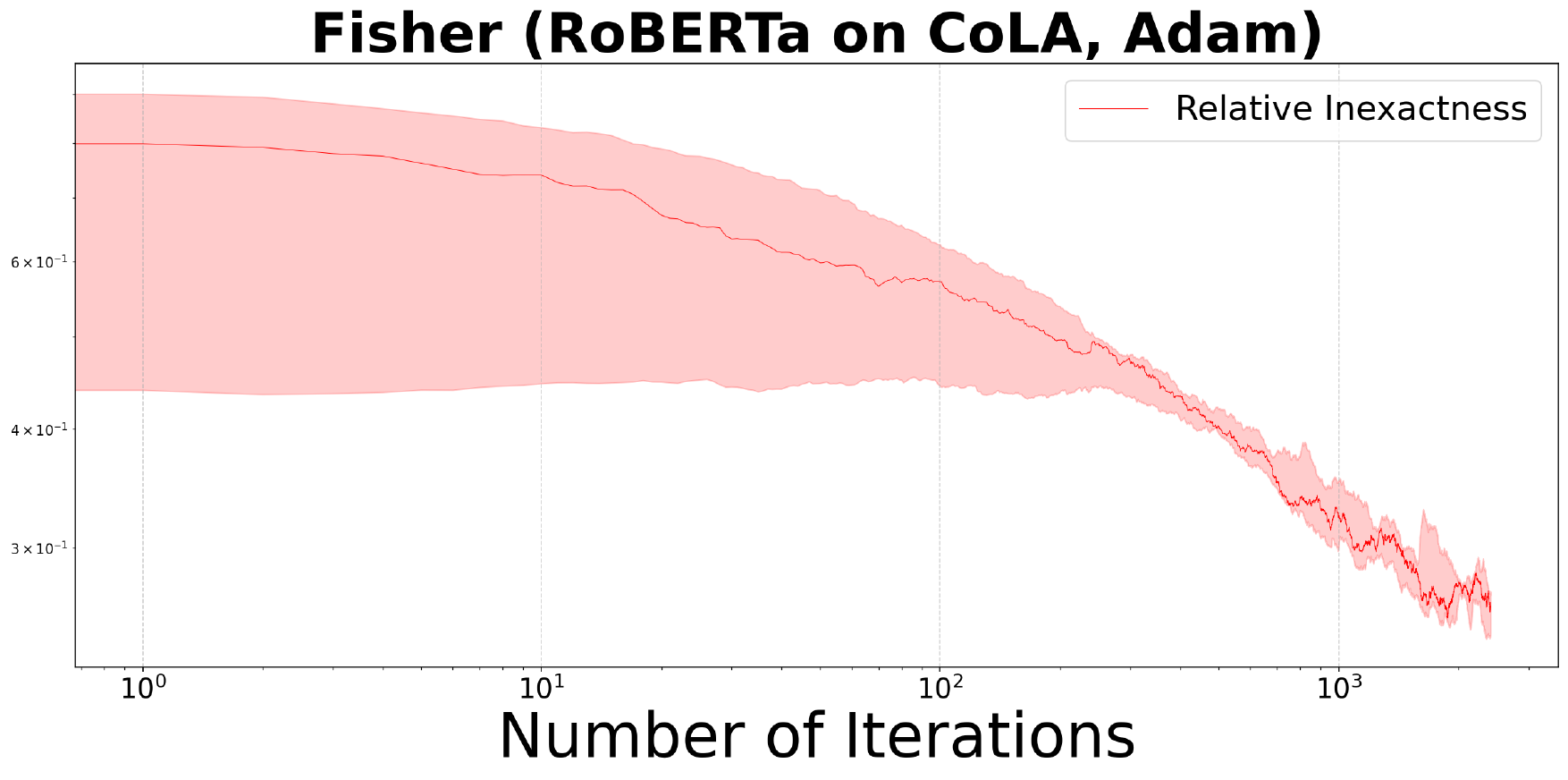}}%
    \hfill
    \subfloat{\label{subfig:hutch_ri_bert}%
      \includegraphics[width=0.3\linewidth]{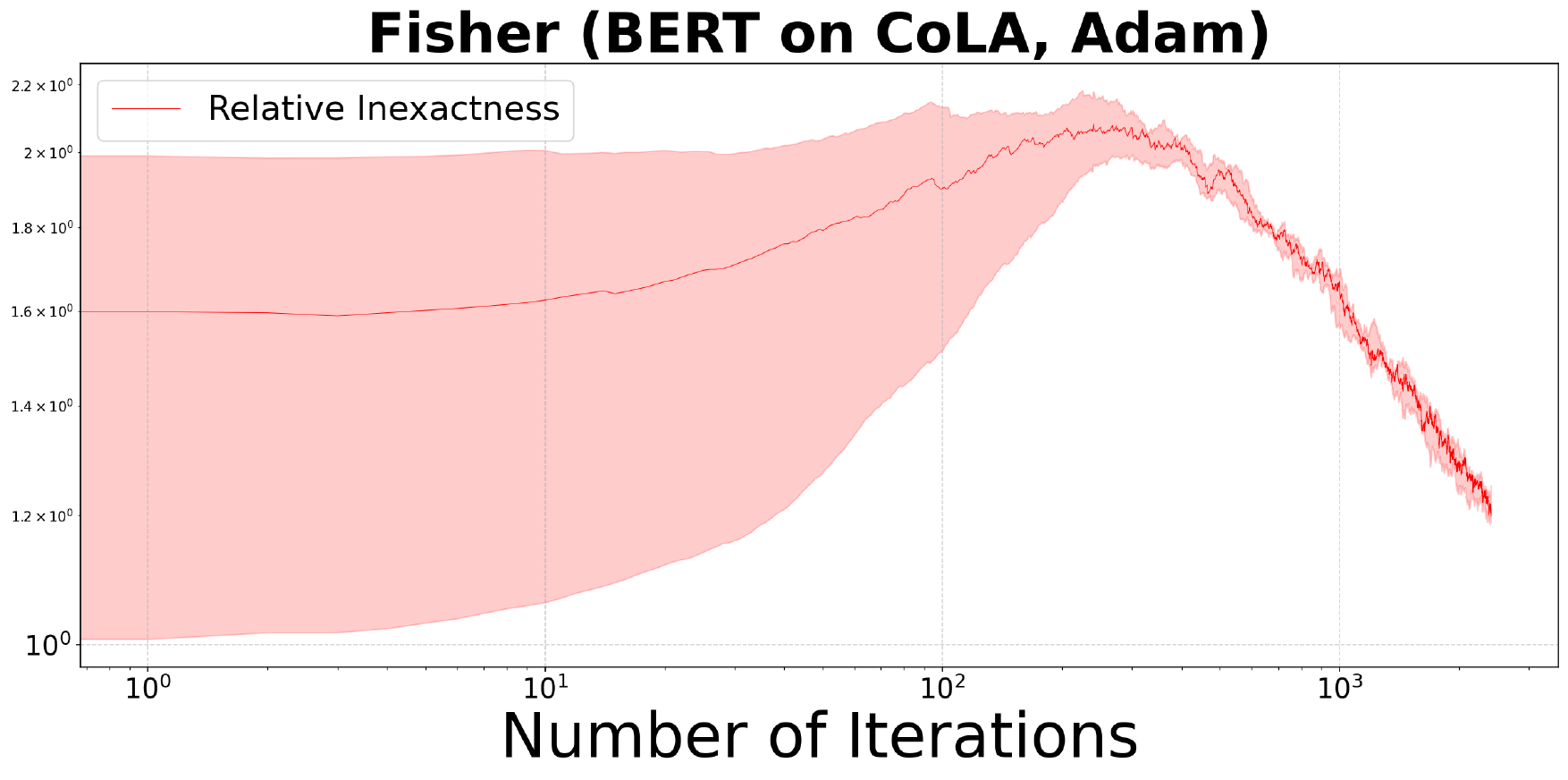}}%
    \hfill
    \subfloat{\label{subfig:fisher_ri_mlp}%
      \includegraphics[width=0.3\linewidth]{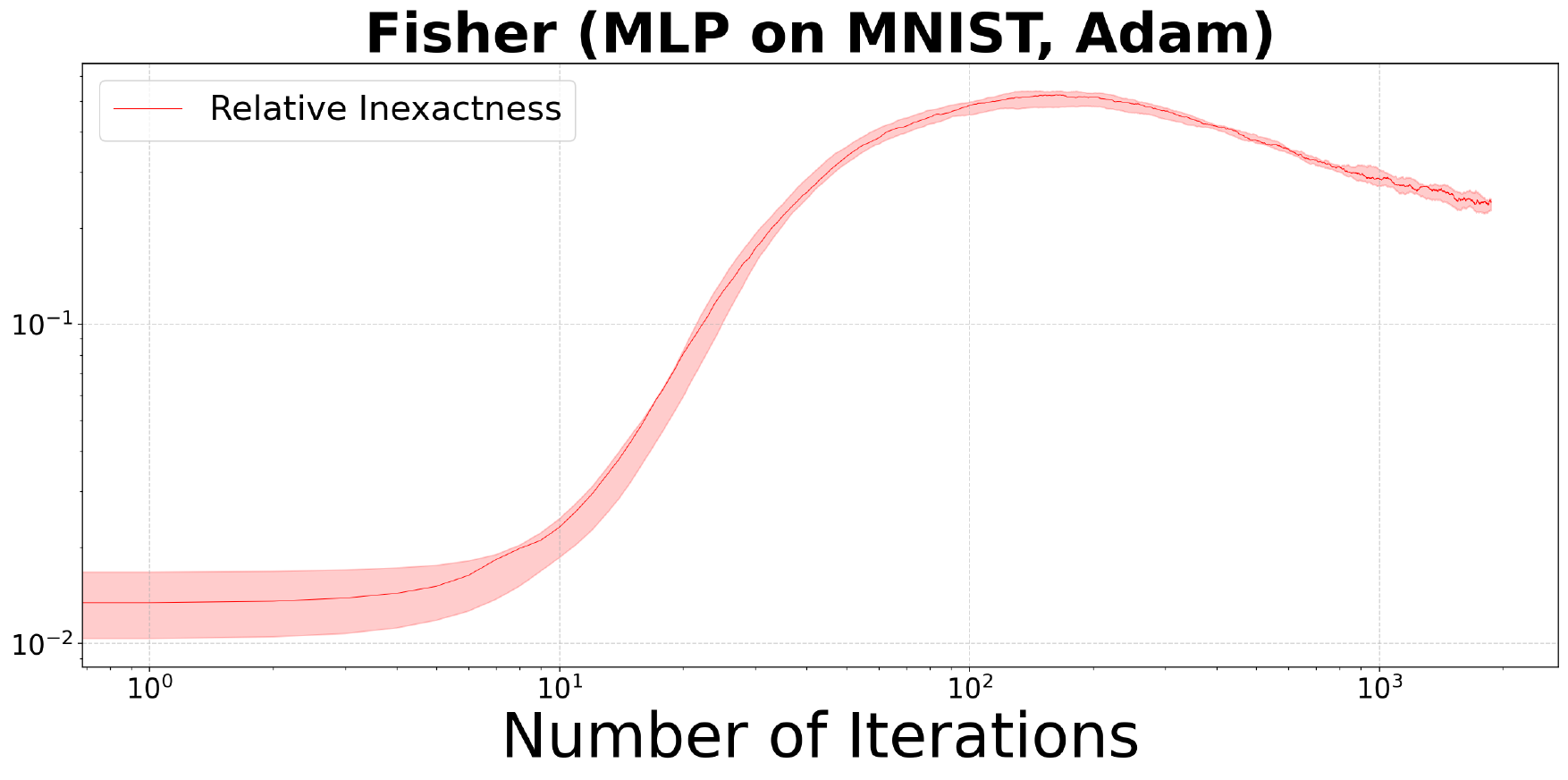}}%
  \caption{$\Delta$ evolution during Supervised fine-tuning task (\textit{first and second column}) and Full training on MNIST (\textit{third column}).}\label{fig:ri}
\end{figure}

We recorded the evolution of the relative inexactness on benchmark tasks that are commonly used in the optimization research literature:

\textbf{Supervised fine-tuning:} we considered a fine-tuning of a widely-used model BERT \cite{devlin2019bertpretrainingdeepbidirectional} and RoBERTa \cite{liu2019robertarobustlyoptimizedbert}, on CoLA \cite{warstadt2019neuralnetworkacceptabilityjudgments} dataset. To make the Hessian computation feasible, we train only the last classification layer and obtain the true Hessian w.r.t. its parameters. The model is trained with the Adam optimizer.

\textbf{Full training:} we train a simple MLP model on the MNIST \cite{mnist} dataset. Hessian and its' approximations are also computed w.r.t. the last layer's parameters. The model is trained with the Adam optimizer.

\textbf{Conclusion:} Seen from Figure \ref{fig:ri}, relative inexactness can be bounded from above by some constant $M$ on various problems and architectures, therefore a wide family of Hessian estimators is appropriate for the \textbf{HAT} algorithm. All the analysis with the derived convergence rates are valid for them.

In Appendix \ref{sec:add_numerical_results}, one can find experiments with relative inexactness for SR-1 and BFGS.

\subsection{HAT performance} In addition, we run our method on \texttt{a9a} dataset from \texttt{LIBSVM} \cite{chang2011libsvm}. We consider the empirical risk on a dataset $\{(a_i,b_i)\}_{i=1}^{n}$ where $a_i \in \mathbb{R}^d$ with two losses: convex and nonconvex. In particular, we use logistic regression ($b_i \in \{-1,+1\}$): $f(x) = \frac{1}{n}\sum_{i=1}^n \log (1 + \exp( - b_i \cdot x^T a_i)), $
and non-linear least squares: $f(x) = \frac{1}{n} \sum\limits_{i=1}^{n} (b_i - 1/(1 + e^{-a_i^T x}))^2.$
The regularization term is the Bregman divergence, generated by $\rho(x) = \frac{1}{2}\<Bx,x>$, with matrix's eigenvalues lying in $[\frac{7}{12}, 1]$.

\begin{figure}[H]
    \subfloat{
      \includegraphics[width=0.45\linewidth]{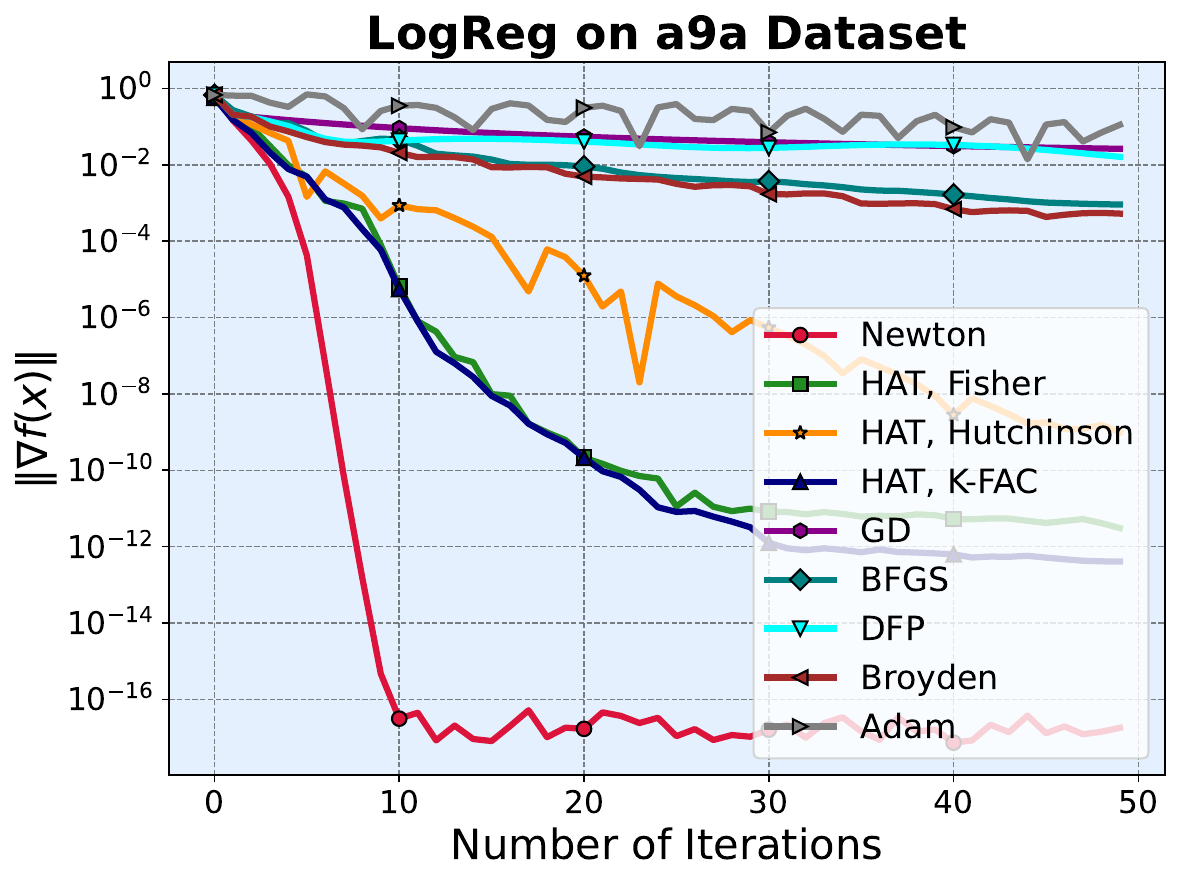}}
    \hfill
    \subfloat{
      \includegraphics[width=0.45\linewidth]{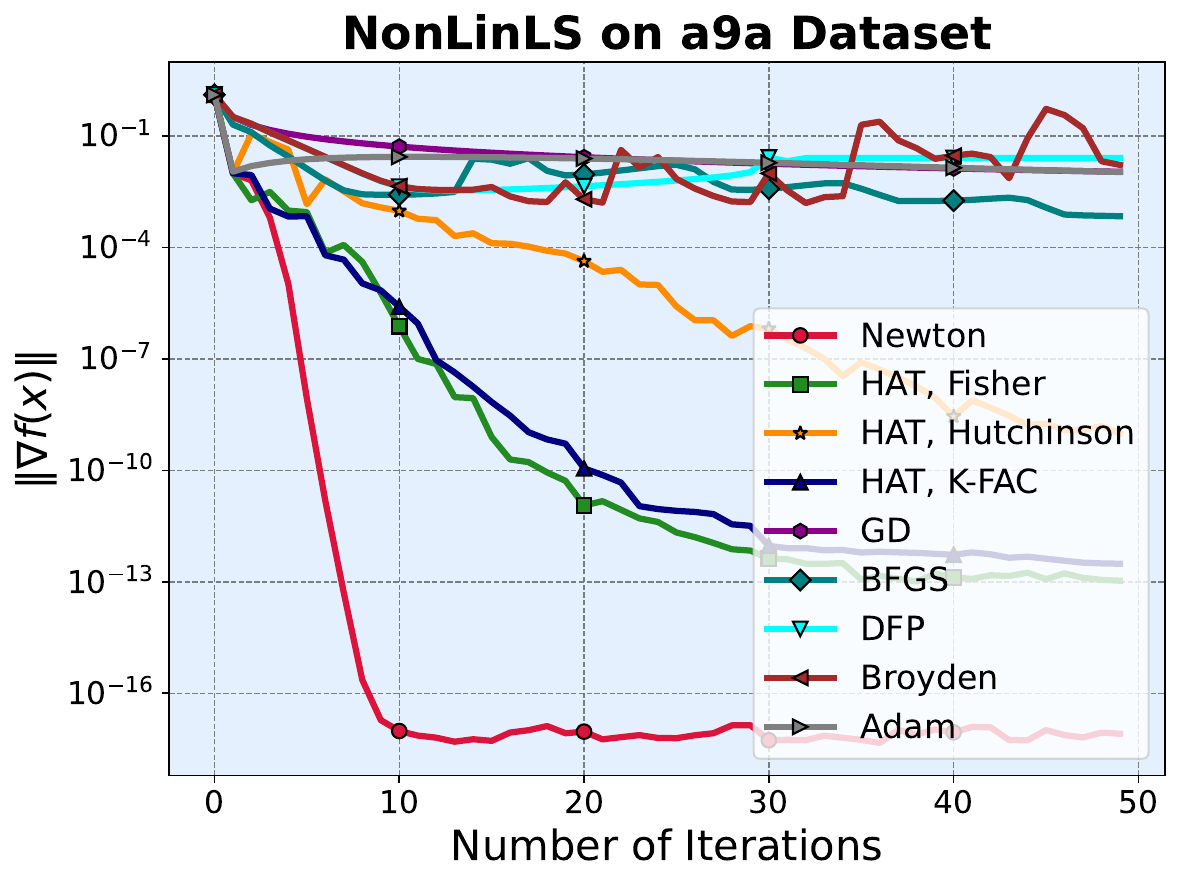}}
  \caption{Comparison of \textbf{HAT} with different methods on \texttt{a9a} dataset}\label{fig:hat_updated}
\end{figure}

We observe, that \textbf{HAT} with Hutchinson and Fisher estimation almost matches the performance of Newton's method on given datasets, though, it consumes significantly less time computing these approximations, instead of the exact Hessian.
For more detailed comparison, please, see Figure \ref{fig:hat_updated} and Figure \ref{fig:hat_nonconvex}. In Appendix \ref{sec:add_numerical_results}, one can find more experiments. All experiments were conducted via A100.

\section{Conclusion}

In this paper we proposed and analyzed the \textbf{Hessian Adaptive Trust-region} (\textbf{HAT}) method, which combines adaptive quadratic regularization, trust-region constraints, and inexact Hessian information within a general Bregman geometry. Our framework unifies and extends several existing lines of work: it allows for rich Hessian approximations (including quasi-Newton, Fisher-type, and diagonal/scaling-based preconditioners), supports a wide class of Bregman divergences, and is provably applicable to both convex and nonconvex problems. At the core of our analysis is the notion of \emph{relative inexactness}, which links the quality of the Hessian approximation to the current gradient norm and naturally captures many practical estimators used in modern large-scale optimization.

There are several promising directions for future work. First, extending \textbf{HAT} to fully stochastic or mini-batch regimes, with noisy gradients and Hessian estimators, would broaden its applicability to large-scale deep learning. Second, our relative inexactness framework suggests new designs of Hessian estimators tailored to specific architectures (e.g., transformers) and problem structures, as well as combinations with lazy or compressed updates in distributed settings. We believe that the proposed \textbf{HAT} framework provides a flexible and theoretically grounded foundation for the next generation of practical second-order methods.

\end{mainpart}
  \bibliographystyle{plainnat}
  \bibliography{references}

@book{goodfellow2016deep,
title={Deep learning},
author={Goodfellow, Ian and Bengio, Yoshua and Courville, Aaron and Bengio, Yoshua},
volume={1},
year={2016},
publisher={MIT Press}
}

@article{nesterov2006cubic,
  title={Cubic regularization of Newton method and its global performance},
  author={Nesterov, Yurii and Polyak, Boris T},
  journal={Mathematical programming},
  volume={108},
  number={1},
  pages={177--205},
  year={2006},
  publisher={Springer}
}

@inproceedings{kantorovich1948newton,
  title={On Newton’s method for functional equations},
  author={Kantorovich, Leonid V},
  booktitle={Dokl. Akad. Nauk SSSR},
  volume={59},
  number={7},
  pages={1237--1240},
  year={1948}
}

@book{nocedal1999numerical,
  title={Numerical optimization},
  author={Nocedal, Jorge and Wright, Stephen J},
  year={1999},
  publisher={Springer}
}

@article{jiang2023universal,
  title={A Universal Trust-Region Method for Convex and Nonconvex Optimization},
  author={Jiang, Yuntian and He, Chang and Zhang, Chuwen and Ge, Dongdong and Jiang, Bo and Ye, Yinyu},
  journal={arXiv preprint arXiv:2311.11489},
  year={2023}
}

@article{Bekas2007AnEF,
  title={An estimator for the diagonal of a matrix},
  author={Costas Bekas and Effrosini Kokiopoulou and Yousef Saad},
  journal={Applied Numerical Mathematics},
  year={2007},
  volume={57},
  pages={1214-1229},
  url={https://api.semanticscholar.org/CorpusID:9023598}
}

@misc{kunstner2020limitationsempiricalfisherapproximation,
      title={Limitations of the Empirical Fisher Approximation for Natural Gradient Descent}, 
      author={Frederik Kunstner and Lukas Balles and Philipp Hennig},
      year={2020},
      eprint={1905.12558},
      archivePrefix={arXiv},
      primaryClass={cs.LG},
      url={https://arxiv.org/abs/1905.12558}, 
}

@misc{chaudhari2017entropysgdbiasinggradientdescent,
      title={Entropy-SGD: Biasing Gradient Descent Into Wide Valleys}, 
      author={Pratik Chaudhari and Anna Choromanska and Stefano Soatto and Yann LeCun and Carlo Baldassi and Christian Borgs and Jennifer Chayes and Levent Sagun and Riccardo Zecchina},
      year={2017},
      eprint={1611.01838},
      archivePrefix={arXiv},
      primaryClass={cs.LG},
      url={https://arxiv.org/abs/1611.01838}, 
}

@misc{devlin2019bertpretrainingdeepbidirectional,
      title={BERT: Pre-training of Deep Bidirectional Transformers for Language Understanding}, 
      author={Jacob Devlin and Ming-Wei Chang and Kenton Lee and Kristina Toutanova},
      year={2019},
      eprint={1810.04805},
      archivePrefix={arXiv},
      primaryClass={cs.CL},
      url={https://arxiv.org/abs/1810.04805}, 
}

@misc{liu2019robertarobustlyoptimizedbert,
      title={RoBERTa: A Robustly Optimized BERT Pretraining Approach}, 
      author={Yinhan Liu and Myle Ott and Naman Goyal and Jingfei Du and Mandar Joshi and Danqi Chen and Omer Levy and Mike Lewis and Luke Zettlemoyer and Veselin Stoyanov},
      year={2019},
      eprint={1907.11692},
      archivePrefix={arXiv},
      primaryClass={cs.CL},
      url={https://arxiv.org/abs/1907.11692}, 
}

@ARTICLE{mnist,
  author={Lecun, Y. and Bottou, L. and Bengio, Y. and Haffner, P.},
  journal={Proceedings of the IEEE}, 
  title={Gradient-based learning applied to document recognition}, 
  year={1998},
  volume={86},
  number={11},
  pages={2278-2324},
  doi={10.1109/5.726791}}

@misc{warstadt2019neuralnetworkacceptabilityjudgments,
      title={Neural Network Acceptability Judgments}, 
      author={Alex Warstadt and Amanpreet Singh and Samuel R. Bowman},
      year={2019},
      eprint={1805.12471},
      archivePrefix={arXiv},
      primaryClass={cs.CL},
      url={https://arxiv.org/abs/1805.12471}, 
}

@misc{jahani2021doublyadaptivescaledalgorithm,
      title={Doubly Adaptive Scaled Algorithm for Machine Learning Using Second-Order Information}, 
      author={Majid Jahani and Sergey Rusakov and Zheng Shi and Peter Richtárik and Michael W. Mahoney and Martin Takáč},
      year={2021},
      eprint={2109.05198},
      archivePrefix={arXiv},
      primaryClass={cs.LG},
      url={https://arxiv.org/abs/2109.05198}, 
}

@article{tieleman2017divide,
  title={Divide the gradient by a running average of its recent magnitude. coursera: Neural networks for machine learning},
  author={Tieleman, Tijmen and Hinton, G},
  journal={Technical report},
  year={2017}
}

@misc{zhang2024adamminiusefewerlearning,
      title={Adam-mini: Use Fewer Learning Rates To Gain More}, 
      author={Yushun Zhang and Congliang Chen and Ziniu Li and Tian Ding and Chenwei Wu and Yinyu Ye and Zhi-Quan Luo and Ruoyu Sun},
      year={2024},
      eprint={2406.16793},
      archivePrefix={arXiv},
      primaryClass={cs.LG},
      url={https://arxiv.org/abs/2406.16793}, 
}

@misc{singh2020woodfisherefficientsecondorderapproximation,
      title={WoodFisher: Efficient Second-Order Approximation for Neural Network Compression}, 
      author={Sidak Pal Singh and Dan Alistarh},
      year={2020},
      eprint={2004.14340},
      archivePrefix={arXiv},
      primaryClass={cs.LG},
      url={https://arxiv.org/abs/2004.14340}, 
}

@article{Davidon1959VariableMM,
  title={Variable Metric Method for Minimization},
  author={William C. Davidon},
  journal={SIAM J. Optim.},
  year={1959},
  volume={1},
  pages={1-17},
  url={https://api.semanticscholar.org/CorpusID:1819475}
}

@inproceedings{Fletcher1988PracticalMO,
  title={Practical Methods of Optimization},
  author={Roger Fletcher},
  year={1988},
  url={https://api.semanticscholar.org/CorpusID:123487779}
}

@article{Conn1991ConvergenceOQ,
  title={Convergence of quasi-Newton matrices generated by the symmetric rank one update},
  author={Andrew R. Conn and Nicholas I. M. Gould and Philippe L. Toint},
  journal={Mathematical Programming},
  year={1991},
  volume={50},
  pages={177-195},
  url={https://api.semanticscholar.org/CorpusID:28028770}
}

@article{Liu1989OnTL,
  title={On the limited memory BFGS method for large scale optimization},
  author={Dong C. Liu and Jorge Nocedal},
  journal={Mathematical Programming},
  year={1989},
  volume={45},
  pages={503-528},
  url={https://api.semanticscholar.org/CorpusID:5681609}
}

@article{Amari1998NaturalGW,
  title={Natural Gradient Works Efficiently in Learning},
  author={Shun‐ichi Amari},
  journal={Neural Computation},
  year={1998},
  volume={10},
  pages={251-276},
  url={https://api.semanticscholar.org/CorpusID:207585383}
}

@article{Amari2021InformationG,
  title={Information geometry},
  author={Shun‐ichi Amari},
  journal={Japanese Journal of Mathematics},
  year={2021},
  volume={16},
  pages={1 - 48},
  url={https://api.semanticscholar.org/CorpusID:230718820}
}

@misc{liu2024sophiascalablestochasticsecondorder,
      title={Sophia: A Scalable Stochastic Second-order Optimizer for Language Model Pre-training}, 
      author={Hong Liu and Zhiyuan Li and David Hall and Percy Liang and Tengyu Ma},
      year={2024},
      eprint={2305.14342},
      archivePrefix={arXiv},
      primaryClass={cs.LG},
      url={https://arxiv.org/abs/2305.14342}, 
}

@inproceedings{doikov2023second,
  title={Second-order optimization with lazy hessians},
  author={Doikov, Nikita and Jaggi, Martin and others},
  booktitle={International Conference on Machine Learning},
  pages={8138--8161},
  year={2023},
  organization={PMLR}
}

@article{duchi2011adaptive,
  title={Adaptive subgradient methods for online learning and stochastic optimization.},
  author={Duchi, John and Hazan, Elad and Singer, Yoram},
  journal={Journal of machine learning research},
  volume={12},
  number={7},
  year={2011}
}

@inproceedings{gupta2018shampoo,
  title={Shampoo: Preconditioned stochastic tensor optimization},
  author={Gupta, Vineet and Koren, Tomer and Singer, Yoram},
  booktitle={International Conference on Machine Learning},
  pages={1842--1850},
  year={2018},
  organization={PMLR}
}

@article{vyas2024soap,
  title={SOAP: Improving and Stabilizing Shampoo using Adam},
  author={Vyas, Nikhil and Morwani, Depen and Zhao, Rosie and Shapira, Itai and Brandfonbrener, David and Janson, Lucas and Kakade, Sham},
  journal={arXiv preprint arXiv:2409.11321},
  year={2024}
}

@article{agafonov2024inexact,
  title={Inexact tensor methods and their application to stochastic convex optimization},
  author={Agafonov, Artem and Kamzolov, Dmitry and Dvurechensky, Pavel and Gasnikov, Alexander and Tak{\'a}{\v{c}}, Martin},
  journal={Optimization Methods and Software},
  volume={39},
  number={1},
  pages={42--83},
  year={2024},
  publisher={Taylor \& Francis}
}

@article{zhang2024transformers,
  title={Why transformers need adam: A hessian perspective},
  author={Zhang, Yushun and Chen, Congliang and Ding, Tian and Li, Ziniu and Sun, Ruoyu and Luo, Zhi-Quan},
  journal={arXiv preprint arXiv:2402.16788},
  year={2024}
}

@inproceedings{islamov2021distributed,
  title={Distributed second order methods with fast rates and compressed communication},
  author={Islamov, Rustem and Qian, Xun and Richt{\'a}rik, Peter},
  booktitle={International conference on machine learning},
  pages={4617--4628},
  year={2021},
  organization={PMLR}
}

@article{safaryan2021fednl,
  title={FedNL: Making Newton-type methods applicable to federated learning},
  author={Safaryan, Mher and Islamov, Rustem and Qian, Xun and Richt{\'a}rik, Peter},
  journal={arXiv preprint arXiv:2106.02969},
  year={2021}
}

@article{doikov2024gradient,
  title={Gradient regularization of Newton method with Bregman distances},
  author={Doikov, Nikita and Nesterov, Yurii},
  journal={Mathematical programming},
  volume={204},
  number={1},
  pages={1--25},
  year={2024},
  publisher={Springer}
}

@article{doikov2024spectral,
  title={Spectral Preconditioning for Gradient Methods on Graded Non-convex Functions},
  author={Doikov, Nikita and Stich, Sebastian U and Jaggi, Martin},
  journal={arXiv preprint arXiv:2402.04843},
  year={2024}
}

@article{nesterov2021implementable,
  title={Implementable tensor methods in unconstrained convex optimization},
  author={Nesterov, Yurii},
  journal={Mathematical Programming},
  volume={186},
  pages={157--183},
  year={2021},
  publisher={Springer}
}

@article{frank1956algorithm,
  title={An algorithm for quadratic programming},
  author={Frank, Marguerite and Wolfe, Philip and others},
  journal={Naval research logistics quarterly},
  volume={3},
  number={1-2},
  pages={95--110},
  year={1956},
  publisher={Wiley Subscription Services, Inc., A Wiley Company New York}
}

@article{polyak1963gradient,
  title={Gradient methods for minimizing functionals},
  author={Polyak, Boris Teodorovich},
  journal={Zhurnal Vychislitel'noi Matematiki i Matematicheskoi Fiziki},
  volume={3},
  number={4},
  pages={643--653},
  year={1963},
  publisher={Russian Academy of Sciences, Branch of Mathematical Sciences}
}

@article{nesterov1983method,
  title={A method of solving a convex programming problem with convergence rate O (1/k** 2)},
  author={Nesterov, Yurii},
  journal={Doklady Akademii Nauk SSSR},
  volume={269},
  number={3},
  pages={543},
  year={1983}
}

@book{nesterov2018lectures,
  title={Lectures on convex optimization},
  author={Nesterov, Yurii and others},
  volume={137},
  year={2018},
  publisher={Springer}
}

@article{tyrrell1970convex,
  title={Convex analysis},
  author={Tyrrell Rockafellar, R},
  journal={Princeton mathematical series},
  volume={28},
  year={1970}
}

@book{liu2011total,
  title={Total Bregman divergence, a robust divergence measure, and its applications},
  author={Liu, Meizhu},
  year={2011},
  publisher={University of Florida}
}

@inproceedings{yao2021adahessian,
  title={Adahessian: An adaptive second order optimizer for machine learning},
  author={Yao, Zhewei and Gholami, Amir and Shen, Sheng and Mustafa, Mustafa and Keutzer, Kurt and Mahoney, Michael},
  booktitle={proceedings of the AAAI conference on artificial intelligence},
  volume={35},
  number={12},
  pages={10665--10673},
  year={2021}
}

@article{kingma2014adam,
  title={Adam: A method for stochastic optimization},
  author={Kingma, Diederik P},
  journal={arXiv preprint arXiv:1412.6980},
  year={2014}
}

@article{chang2011libsvm,
  title={LIBSVM: A library for support vector machines},
  author={Chang, Chih-Chung and Lin, Chih-Jen},
  journal={ACM transactions on intelligent systems and technology (TIST)},
  volume={2},
  number={3},
  pages={1--27},
  year={2011},
  publisher={Acm New York, NY, USA}
}

@article{dauphin2014identifying,
  title={Identifying and attacking the saddle point problem in high-dimensional non-convex optimization},
  author={Dauphin, Yann N and Pascanu, Razvan and Gulcehre, Caglar and Cho, Kyunghyun and Ganguli, Surya and Bengio, Yoshua},
  journal={Advances in neural information processing systems},
  volume={27},
  year={2014}
}

@article{Broyden1965ACO,
  title={A Class of Methods for Solving Nonlinear Simultaneous Equations},
  author={C. G. Broyden},
  journal={Mathematics of Computation},
  year={1965},
  volume={19},
  pages={577-593},
  url={https://api.semanticscholar.org/CorpusID:2802972}
}

@inproceedings{martens2015optimizing,
  title={Optimizing neural networks with kronecker-factored approximate curvature},
  author={Martens, James and Grosse, Roger},
  booktitle={International conference on machine learning},
  pages={2408--2417},
  year={2015},
  organization={PMLR}
}

@article{doikov2023first,
  title={First and zeroth-order implementations of the regularized Newton method with lazy approximated Hessians},
  author={Doikov, Nikita and Grapiglia, Geovani Nunes},
  journal={arXiv preprint arXiv:2309.02412},
  year={2023}
}

@article{roosta2019sub,
  title={Sub-sampled Newton methods},
  author={Roosta-Khorasani, Farbod and Mahoney, Michael W},
  journal={Mathematical Programming},
  volume={174},
  number={1},
  pages={293--326},
  year={2019},
  publisher={Springer}
}

@article{pilanci2017newton,
  title={Newton sketch: A near linear-time optimization algorithm with linear-quadratic convergence},
  author={Pilanci, Mert and Wainwright, Martin J},
  journal={SIAM Journal on Optimization},
  volume={27},
  number={1},
  pages={205--245},
  year={2017},
  publisher={SIAM}
}

@article{agarwal2017second,
  title={Second-order stochastic optimization for machine learning in linear time},
  author={Agarwal, Naman and Bullins, Brian and Hazan, Elad},
  journal={Journal of Machine Learning Research},
  volume={18},
  number={116},
  pages={1--40},
  year={2017}
}
\begin{appendixpart}

\section{Missing proofs}\label{sec:appendix}

\begin{lemma}(Karush-Kuhn-Tucker conditions)\label{lem:kkt}
Let $d_k$ be the solution to the subproblem (\ref{alg:subproblem}), then there exists $\lambda \geq 0$, such that the following expressions take place:
\begin{align*}
    &\|d_k\|^2 \leq r_k\|\nabla f(x_k)\|^{1/2}\\
    &\lambda\left(\|d_k\| - r_k\|\nabla f(x_k)\|^{1/2}\right) = 0\\
    &H_k + A_k\nabla^2\rho(x_k+d_k) + \lambda I \succcurlyeq 0\\
    &(H_k+\lambda I) + A_k\left(\nabla\rho(x_k+d_k) - \nabla\rho(x_k)\right) = -\nabla f(x_k)
\end{align*}
Proof is provided in Theorem 4.1 from \cite{nocedal1999numerical}.
\end{lemma}

\begin{lemma}\label{lem:rho}
    Let Assumption \ref{as:rho} be satisfied, then for $\forall x, d \in \RR^n$
    \[\sigma_V \|d\|^2 \leq \<\nabla\rho(x+d) - \nabla\rho(x),d> \leq L_V\|d\|^2.\]
\end{lemma}    
\begin{proof}
    Write the definition of strong convexity from Assumption \ref{as:rho} with $x$ and $x+d$:
   \begin{eqnarray*}
        \rho(x+d) &\geq& \rho(x) + \<\nabla \rho(x), d> + \frac{\sigma_V}{2}\|d\|^2,\\
        \rho(x) &\geq& \rho(x+d) + \<\nabla \rho(x+d), -d> + \frac{\sigma_V}{2}\|d\|^2.
    \end{eqnarray*}
    Summing two inequalities and acquiring, we get:
    \[\<\nabla\rho(x+d) - \nabla\rho(x),d> \geq \sigma_V\|d\|^2.\]
    In a similar way, we obtain an upper bound:
    \[\<\nabla\rho(x+d) - \nabla\rho(x),d> \leq L_V\|d\|^2.\]
\end{proof}

\begin{lemma}\label{lem:val_decr}
    Let Assumptions \ref{as:heslip}, \ref{as:rho} be satisfied. If $d$ is a solution to the subproblem (\ref{alg:subproblem}) and Lagrange multiplier $\lambda > 0$, then
    \begin{eqnarray*}
    f(x_k + d) - f(x_k) \leq &-&\frac{1}{2}\Bigg(\frac{\lambda}{r_k\|\nabla f(x_k)\|^{1/2}} + \frac{(2\sigma_V - L_V) A_k}{r_k\|\nabla f(x_k)\|^{1/2}}\\
    &-& \frac{L_2}{3} - \frac{\|\nabla^2f(x_k) - H_k\|}{r_k\|\nabla f(x_k)\|^{1/2}}\Bigg)r_k^3\|\nabla f(x_k)\|^{3/2}.
    \end{eqnarray*}
\end{lemma}    
\begin{proof}
    From Assumption \ref{as:heslip} and Lemma \ref{lem:nest}, we can derive
    \begin{eqnarray*}
        f(x_k+d) &\leq& f(x_k) + \<\nabla f(x_k),d> + \frac{1}{2}\<\nabla^2 f(x_k)d,d> + \frac{L_2}{6}\|d\|^3\\
        &=&f(x_k) - \<H_kd,d> - \lambda\|d\|^2 - A_k\<\nabla\rho(x_k+d)-\nabla\rho(x),d> \\
        &+&\frac{1}{2}\<\nabla^2 f(x_k)d,d> + \frac{L_2}{6}\|d\|^3\\
        &\stackrel{(L.\ref{lem:rho})}{\leq}&f(x_k) - \<H_kd,d> - \lambda\|d\|^2 - \sigma_VA_k\|d\|^2\\
        &+&\frac{1}{2}\<\nabla^2 f(x_k)d,d> + \frac{L_2}{6}\|d\|^3\\
\end{eqnarray*}
After adding and extracting $\frac{1}{2}\<H_k d,d>$ we get
\begin{eqnarray*}
        f(x_k+d)&=&f(x_k) - \frac{1}{2}\<H_kd,d> - \lambda\|d\|^2 - \sigma_VA_k\|d\|^2 \\
        &+& \frac{1}{2}\<\left(\nabla^2 f(x_k)-H_k\right)d,d> + \frac{L_2}{6}\|d\|^3\\
        &\leq& f(x_k) - \frac{1}{2}\<H_kd,d> - \lambda\|d\|^2 - \sigma_VA_k\|d\|^2 \\
        &+&\frac{1}{2}\|\nabla^2 f(x_k) - H_k\|\cdot \|d\|^2+ \frac{L_2}{6}\|d\|^3\\
        &\stackrel{(L.\ref{lem:kkt})}{\leq}& f(x_k) + \frac{A_k}{2}\<\nabla^2\rho(x_k+d)d,d> \\
        &-& \frac{1}{2}\left(\lambda + 2\sigma_VA_k - \|\nabla^2 f(x_k) - H_k\|\right)\|d\|^2 + \frac{L_2}{6}\|d\|^3 \\
        &\stackrel{(L.\ref{lem:rho})}{\leq}& f(x_k) \\
        &-& \frac{1}{2}\left(\lambda + 2\sigma_VA_k - \|\nabla^2 f(x_k) - H_k\| - L_VA_k\right)\|d\|^2 + \frac{L_2}{6}\|d\|^3\\
        &=& f(x_k) - \frac{1}{2}\left(\frac{\lambda}{\|d\|} + \frac{(2\sigma_V - L_V)A_k}{\|d\|} - \frac{L_2}{3} - \frac{\|\nabla^2 f(x_k) - H_k\|}{\|d\|}\right)\|d\|^3.
    \end{eqnarray*}
    As $\lambda > 0$, then, due to the KKT conditions $\|d\| = r_k\|\nabla f(x_k)\|^{1/2}$, we have
     \begin{eqnarray*}
    f(x_k + d) - f(x_k) \leq &-&\frac{1}{2}\Bigg(\frac{\lambda}{r_k\|\nabla f(x_k)\|^{1/2}} + \frac{(2\sigma_V - L_V) A_k}{r_k\|\nabla f(x_k)\|^{1/2}} \\
    &-& \frac{L_2}{3} - \frac{\|\nabla^2f(x_k) - H_k\|}{r_k\|\nabla f(x_k)\|^{1/2}}\Bigg)r_k^3\|\nabla f(x_k)\|^{3/2}.
    \end{eqnarray*}
\end{proof}
\begin{lemma}\label{lem:grad_decr}
    Let Assumptions \ref{as:heslip}, \ref{as:rho} be satisfied. If $d$ is a solution to the subproblem (\ref{alg:subproblem}) and Lagrange multiplier $\lambda = 0$, then
    \[\|\nabla f(x_k+d)\| \leq \left(\frac{L_2}{2}r_k^2 + \frac{L_V A_kr_k}{\|\nabla f(x_k)\|^{1/2}} + \frac{\|\nabla^2f(x_k)-H_k\|r_k}{\|\nabla f(x_k)\|^{1/2}}\right)\|\nabla f(x_k)\|.\]
\end{lemma}    
\begin{proof}
    Using Assumption \ref{as:heslip} and Lemma \ref{lem:nest}, we obtain:
    \begin{eqnarray*}
        \|\nabla f(x_k+d)\| &=& \|\nabla f(x_k+d) - \nabla f(x_k) - \nabla^2f(x_k)d + \nabla f(x_k) + \nabla^2f(x_k)d\|\\
        &\leq& \|\nabla f(x_k+d) - \nabla f(x_k) - \nabla^2f(x_k)d\| + \| \nabla f(x_k) + \nabla^2f(x_k)d\|\\
        &\leq& \frac{L_2}{2}\|d\|^2 + \|\nabla f(x_k) + \nabla^2f(x_k)d -H_kd + H_kd\| \\
        &\stackrel{(L.\ref{lem:kkt})}{\leq}& \frac{L_2}{2}\|d\|^2 + \|\nabla f(x_k)d + H_kd\| + \|\nabla^2 f(x_k)d - H_kd\|\\
        &=&  \frac{L_2}{2}\|d\|^2 + \|\lambda d + A_k(\nabla\rho(x_k+d)-\nabla\rho(x_k))\| + \|\nabla^2 f(x_k) - H_k\| \|d\|.
    \end{eqnarray*}    
    As $\lambda = 0$ and $\rho$ has $L_V$-Lipschitz gradient, therefore
    \begin{eqnarray*}
        \|\nabla f(x_k+d)\| &\leq& \frac{L_2}{2}\|d\|^2 +L_VA_k\|d\| + \|\nabla^2 f(x_k) - H_k\|\cdot \|d\|\\
        &\leq& \frac{L_2}{2}r_k^2\|\nabla f(x_k)\| +L_VA_kr_k\|\nabla f(x_k)\|^{1/2} \\
        &+& \|\nabla^2 f(x_k) - H_k\|r_k\|\nabla f(x_k)\|^{1/2}\\
        &=&\left(\frac{L_2}{2}r_k^2 + \frac{L_VA_kr_k}{\|\nabla f(x_k\|^{1/2})} + \frac{\|\nabla^2f(x_k)-H_k\|r_k}{\|\nabla f(x_k)\|^{1/2}}\right) \|\nabla f(x_k)\|.
    \end{eqnarray*}
\end{proof}
Now that we have derived all the needed inequalities, we can choose parameters $r_k$ and $A_k$, depending on $\xi$ and $\eta$, that the solution of the subproblem ($\ref{alg:subproblem}$) satisfies necessary conditions for the convergence.

\begin{theorem}\label{th:selection}\textup{(Theorem \ref{theor})}
    Let Assumptions \ref{as:heslip}, \ref{as:rho} be satisfied. If parameters $r_k$ and $A_k$ are chosen as
    \begin{eqnarray*}
    r_k &=& \frac{\xi}{\frac{\|\nabla ^2 f(x_k) - H_k\|}{\|\nabla f(x_k)\|^{1/2}}\left(1 + \frac{L_V}{2\sigma_V-L_V}\right) + \sqrt{\xi\left(\frac{L_2}{2} + L_V\frac{2\eta + L_2/3}{2\sigma_V-L_V}\right)}}\\
    A_k &=& \frac{\|\nabla ^2 f(x_k) - H_k\|}{2\sigma_V-L_V} + \frac{r_k\|\nabla f(x_k)\|^{1/2}}{2\sigma_V - L_V}\left(2\eta + \frac{L_2}{3}\right),
\end{eqnarray*}
    then the inequalities
    \begin{equation*}
    f(x_k+d_k) - f(x_k) \leq -\eta r_k^3\|\nabla f(x_k)\|^{3/2} ~~~\text{and}~~~ \|\nabla f(x_k+d_k)\| \leq \xi\|\nabla f(x_k)\|
\end{equation*}
    hold.
\end{theorem}
\begin{proof}
    Denote $d$ as a solution to the subproblem (\ref{alg:subproblem}) on step $k$. If $\lambda >0$, then, according to Lemma \ref{lem:val_decr}, we have
    \begin{eqnarray*}
    f(x_k + d) - f(x_k) \leq &-&\frac{1}{2}\Bigg(\frac{\lambda}{r_k\|\nabla f(x_k)\|^{1/2}} + \frac{(2\sigma_V - L_V) A_k}{r_k\|\nabla f(x_k)\|^{1/2}} \\
    &-& \frac{L_2}{3} - \frac{\|\nabla^2f(x_k) - H_k\|}{r_k\|\nabla f(x_k)\|^{1/2}}\Bigg)r_k^3\|\nabla f(x_k)\|^{3/2}.
    \end{eqnarray*}
    If $\lambda = 0,$ then according to Lemma \ref{lem:grad_decr}, we have
    \begin{equation*}
    \|\nabla f(x_k+d)\| \leq \left(\frac{L_2}{2}r_k^2 + \frac{L_V A_kr_k}{\|\nabla f(x_k)\|^{1/2}} + \frac{\|\nabla^2f(x_k)-H_k\|r_k}{\|\nabla f(x_k)\|^{1/2}}\right)\|\nabla f(x_k)\|.
    \end{equation*}
    When $\lambda > 0$ we bound the term from 0, by setting 
    \begin{equation*}
        \frac{(2\sigma_V - L_V) A_k}{r_k\|\nabla f(x_k)\|^{1/2}} - \frac{L_2}{3} -\frac{\|\nabla^2f(x_k) - H_k\|}{r_k\|\nabla f(x_k)\|^{1/2}} = 2\eta.
    \end{equation*}
    We can express $A_k$ from above, therefore
    \begin{equation*}
        A_k = \frac{\|\nabla ^2 f(x_k) - H_k\|}{2\sigma_V-L_V} + \frac{r_k\|\nabla f(x_k)\|^{1/2}}{2\sigma_V - L_V}\left(2\eta + \frac{L_2}{3}\right).
    \end{equation*}
    Substituting this $A_k$ in the case, with $\lambda = 0$, we get
    \begin{equation*}
        \left(\frac{L_2}{2} + L_V\frac{2\eta + \frac{L_2}{3}}{2\sigma_V-L_V}\right)r_k^2 + \left(\frac{\|\nabla^2f(x_k)-H_k\|}{\|\nabla f(x_k)\|^{1/2}} + L_V\frac{\|\nabla^2f(x_k)-H_k\|}{(2\sigma_V-L_V)\|\nabla f(x_k)\|^{1/2}}\right)r_k \leq \xi
    \end{equation*}
    for some $\xi < 1$. Next, we divide both parts by $\xi$ and use the fact, that for inequality $ax^2 + bx \leq 1$ one can take $x = \frac{1}{b + \sqrt{a}}$. Therefore, we can choose
    \begin{equation*}
        r_k = \frac{1}{\frac{\Delta_k}{\xi} + \sqrt{\frac{a}{\xi}}} = \frac{\xi}{\Delta_k + \sqrt{a\xi}}
    \end{equation*}
    with 
    \begin{eqnarray*}
        \Delta_k &=& \frac{\|\nabla^2f(x_k)-H_k\|}{\|\nabla f(x_k)\|^{1/2}}\left(1 + \frac{L_V}{2\sigma_V-L_V}\right)\\
        a &=& \frac{L_2}{2} + L_V\frac{2\eta + \frac{L_2}{3}}{2\sigma_V-L_V}.
    \end{eqnarray*}
\end{proof}
\begin{lemma}\label{lem:monotone}
    Let Assumptions \ref{as:heslip}, \ref{as:rho} be satisfied. With selection of $r_k$ and $A_k$ from Theorem \ref{th:selection}, we have $f(x_{k+1}) \leq f(x_k)$ for all $k \in \mathbb{N}$.
\end{lemma}
\begin{proof}
    If $d$ is a solution to the subproblem (\ref{alg:subproblem}), then, due to  Lemma \ref{lem:val_decr}:
    \begin{eqnarray*}
    f(x_{k+1}) = f(x_k + d) &\leq& f(x_k) - \frac{1}{2}\left(\lambda  + (2\sigma_V-L_V)A_k - \left\|\nabla^2 f(x_k)-H_k\right\|\right)\|d\|^2 \\
    &+& \frac{L_2}{3}\|d\|^3\\
    &=&f(x_k) - \frac{1}{2}\left(\lambda + \left(2\eta + \frac{L_2}{3}\right)r_k\|\nabla f(x_k)\|^{1/2}\right)\|d\|^2 + \frac{L_2}{3}\|d\|^3\\
    &\stackrel{(L.\ref{lem:kkt})}{\leq}&f(x_k) - \frac{1}{2}\left(\lambda + \left(2\eta + \frac{L_2}{3}\right)\|d\|\right)\|d\|^2 + \frac{L_2}{3}\|d\|^3\\
    &=&f(x_k) -\frac{\lambda}{2}\|d\|^2 -  \left(\eta + \frac{L_2}{6} - \frac{L_2}{6}\right)\|d\|^3 \\
    &=&  f(x_k)-\frac{\lambda}{2}\|d\|^2 - \eta\|d\|^3 \leq f(x_k).
    \end{eqnarray*}
\end{proof}
\begin{remark}
Here and below while discussing iterations with sufficient decrease in the function's value: 
$$f(x_{k+1}) - f(x_k) \leq -\eta r_k^3\|\nabla f(x_k)\|^{3/2},$$
we will write
$$f(x_{k+1}) - f(x_k) \leq -\eta\|\nabla f(x_k)\|^{3/2}.$$
The intuition is the following: $r_k$ is bounded from below by some $r_{min}$, therefore after replacing $\eta := \eta r_{min}^3$ all inequalities stay valid.
\end{remark}
\begin{lemma}\label{lem::set_size}
    Let Assumption \ref{as:sublevel} be satisfied. Suppose, that for every iteration $k$ of Algorithm \ref{alg:alg1} the inequalities (\ref{eq:iter1}) and (\ref{eq:iter2}) hold and $f(x_{k}) \leq f(x_{k-1})$. Denote $j_f$ as the first iteration, when we obtain $\|\nabla f(x_{j_f})\| \leq \e$. Then for sets
    \[\FF = \{k < j_f ~:~ f(x_k+d_k) - f(x_k) \leq -\eta\|\nabla f(x_k)\|^{3/2}\},\]
    \[\GG = \{k < j_f ~:~ \|\nabla f(x_k+d_k)\|\leq \xi \|\nabla f(x_k)\|\},\]
    we have the following bounds on their size:
    \[|\FF| \leq \frac{1}{\eta}(f(x_0) - f_*)\e^{-3/2},\]
    \[|\GG| \leq \left(|\FF| + 1\right)\left\lceil\log_\xi\frac{\e}{G}\right\rceil.\]
\end{lemma}
    \begin{proof}
        Note, that for every $k \in \FF$ we have $\|\nabla f(x_k)\| \geq \e$. Hence,
        \begin{equation*}
            |\FF| = \sum\limits_{k\in\FF}1 \leq \sum\limits_{k\in\FF} \frac{\|\nabla f(x_k)\|^{3/2}}{\e^{3/2}}.
        \end{equation*}
        After using the definition of the set $\FF$ and the fact, that all the iterations $x_k \in \FF$, we obtain
        \[\sum\limits_{k\in\FF}\|\nabla f(x_k)\|^{3/2} \leq \sum\limits_{k\in\FF}\frac{f(x_k) - f(x_{k+1})}{\eta}.\]
        Note, that $f(x_{k+1}) \leq f(x_k)~~\forall k \in \mathbb{N}$, even if these iterations do not lie in $\FF$. Therefore, after numerating all the iterations in $\FF$ as $k(1), k(2), \ldots, k(|\FF|)$ we have $f(x_{k(i)}) \leq f(x_{k(i-1)})$ and
        \[\sum\limits_{k\in\FF}f(x_k) - f(x_{k+1}) \leq f(x_0) - f_*.\]
        Combining all inequalities above, we get
        \[|\FF| = \sum\limits_{k\in\FF}1 \leq \sum\limits_{k\in\FF} \frac{\|\nabla f(x_k)\|^{3/2}}{\e^{3/2}} \leq \e^{-3/2}\sum\limits_{k\in\FF}\frac{f(x_k) - f(x_{k+1})}{\eta} \leq \frac{1}{\eta}(f(x_0) - f_*)\e^{-3/2}.\]
        Then we derive the bound on the size of $\GG$. Note, that if the iteration $k$ ends up in $\GG$, then for any $n > 0$, such that $k+n < j_f$ and iterations $k, k+1, \ldots, k+n$ lie in $\GG$, we have
         \[\e \leq \|\nabla f(x_{k+n})\| \leq \xi^n\|\nabla f(x_k)\| \leq \xi^n G.\]
         As a result we have an upper bound on iterations in $\GG$ in a row $n_{max}$. Then,
         \begin{equation*}
             |\GG| \leq \left(|\FF| + 1\right)n_{max} \leq  \left(|\FF| + 1\right)\left\lceil\log_{\xi}\frac{\e}{G}\right\rceil.
         \end{equation*}
    \end{proof}
\begin{theorem}\textup{(Theorem \ref{th::nonconvex})}
    Let Assumptions \ref{as:heslip}, \ref{as:sublevel}, \ref{as:rho} be satisfied. Then it takes
    \[T = \OO\left(\frac{1}{\eta}(f(x_0)-f_*)\e^{-3/2}\log_\xi\frac{\e}{G}\right)\]
    iterations to obtain an $\e$-approximate first-order stationary point.
\end{theorem}
     \begin{proof}
        Choose parameters $r_k$ and $A_k$ as in Theorem \ref{th:selection}. 
        As proven above, all iterations can be divided into ones with $\lambda > 0$ and $\lambda = 0$. Considering the first iteration $k$, when $\|\nabla f(x_k)\| \leq \e$ we need to conduct $|\FF| + |\GG|$ iterations. According to Theorem \ref{th:selection} and Lemma \ref{lem:monotone} the conditions of Lemma \ref{lem::set_size} are met, therefore we get the needed result.
     \end{proof}

\begin{lemma}\label{lem:conv_size} Let Assumptions \ref{as:conv}, \ref{as:heslip}, \ref{as:sublevel}, \ref{as:rho} be satisfied and for every iteration of Algorithm \ref{alg:alg1}, and the inequalities (\ref{eq:iter1}) and (\ref{eq:iter2}) hold. Denote $j_f$ as the first iteration, when we obtain $f(x_{j_f}) - f_* \leq \e$ and $\FF$ as in Lemma \ref{lem::set_size}. Then,
     \[|\FF| \leq \sqrt{\frac{4D^3}{\e\eta^2}}.\]
\end{lemma}
     \begin{proof}
         Since $f$ is convex, then
         \begin{equation*}
             f_* \geq f(x) + \<\nabla f(x), x_*-x>.
         \end{equation*}
         Rearranging the terms and using the boundedness, we get
         \begin{eqnarray*}
             f(x) - f_* \leq \<\nabla f(x), x-x_*> \leq \|\nabla f(x)\|D.
         \end{eqnarray*}
         Then, if the iteration $k$ belongs to the set $\FF$, we obtain
         \begin{equation*}
             f(x_{k+1}) - f(x_k) \leq -\eta \|\nabla f(x_k)\|^{3/2} \leq -\eta \left(\frac{f(x_k) - f_*}{D}\right)^{3/2}.
         \end{equation*}
         Let us numerate all iterations in $\FF$ as $k(1), k(2), \ldots, k(|\FF|)$, then, denoting $\delta_i = f(x_{k(i)})-f_*$, we get
         \begin{eqnarray*}
             \frac{1}{\sqrt{\delta_{i+1}}} - \frac{1}{\sqrt{\delta_i}} = \frac{\sqrt{\delta_i} - \sqrt{\delta_{i+1}}}{\sqrt{\delta_i}\sqrt{\delta_{i+1}}} = \frac{\delta_i - \delta_{i+1}}{\sqrt{\delta_i}\sqrt{\delta_{i+1}}(\sqrt{\delta_i} + \sqrt{\delta_{i+1}})}.
         \end{eqnarray*}
         From Lemma \ref{lem:monotone} we have for any $k$: $f(x_{k+1}) \leq f(x_k)$, hence $\delta_{i+1} \leq \delta_i$, and, moreover, $\delta_{i} - \delta_{i+1} \geq f(x_{k(i)}) - f(x_{k(i)+1})$. Therefore,
         \begin{eqnarray*}
             \frac{1}{\sqrt{\delta_{i+1}}} - \frac{1}{\sqrt{\delta_i}} \geq \frac{\eta}{D^{3/2}}\frac{\delta_{i}^{3/2}}{\sqrt{\delta_i}\sqrt{\delta_{i+1}}(\sqrt{\delta_i} + \sqrt{\delta_{i+1}})} \geq \frac{\eta}{D^{3/2}}\frac{\delta_{i}^{3/2}}{\sqrt{\delta_i}\sqrt{\delta_{i}}(\sqrt{\delta_i} + \sqrt{\delta_{i}})} = \frac{\eta}{2D^{3/2}}.
         \end{eqnarray*}
         Summing for all iterations, we have for all $k$
         \begin{eqnarray*}
             \frac{1}{\sqrt{\delta_{k}}} - \frac{1}{\sqrt{\delta_0}} \geq \frac{k\eta}{2D^{3/2}}.
         \end{eqnarray*}
         Since $\delta_k \geq \e$, then
         \[k \leq \sqrt{\frac{4D^3}{\e\eta^2}}\]
     \end{proof}

\begin{lemma}\label{lem:conv_g}
    Let Assumptions \ref{as:heslip}, \ref{as:rho} be satisfied, for every iteration of Algorithm \ref{alg:alg1} the inequalities (\ref{eq:iter1}) and (\ref{eq:iter2}) hold, $H_k \succcurlyeq (L_V - \sigma_V)A_k\cdot I$ and $\xi \leq \tfrac{\sqrt{5}-1}{2}$. Then for for all $k\in \mathbb{N}$ we have
    \[\|\nabla f(x_k + d)\| \leq 1/\xi \|\nabla f(x_k)\|.\]
\end{lemma}
    \begin{proof}
        As previously, we have
        \begin{eqnarray*}
        \|\nabla f(x_k + d)\| &\leq& \frac{L_2}{2}\|d\|^2 + \|\lambda d + A_k\left(\nabla \rho(x_k+d) - \nabla \rho(x_k)\right)\| + \|\nabla^2 f(x_k) - H_k\|\cdot \|d\|\\
        &\leq&\frac{L_2}{2}\|d\|^2 + \frac{\lambda \|d\|^2 + L_VA_k\|d\|^2}{\|d\|} + \|\nabla^2 f(x_k) - H_k\|\cdot \|d\|.
        \end{eqnarray*}
        From the KKT conditions we have
        \[\<\nabla f(x_k), d> = \<H_kd + \lambda d + A_k\left(\nabla \rho(x_k + d) - \nabla \rho(x_k)\right),d> \geq \lambda \|d\|^2 + A_k\sigma_V\|d\|^2 + \<H_kd,d>,\]
        while from the Cauchy-Schwarz inequality, we have
        \[\<\nabla f(x_k),d> \leq \|\nabla f(x_k)\|\cdot \|d\|.\]
        Combining this inequalities, we get
        \begin{eqnarray*}
             \|\nabla f(x_k + d)\| &\leq& \frac{L_2}{2}\|d\|^2 +\frac{\lambda \|d\|^2 + L_VA_k\|d\|^2}{\lambda \|d\|^2 + \sigma_VA_k\|d\|^2 + \<H_kd,d>}\|\nabla f(x_k)\|\\
             &+& \|\nabla^2 f(x_k) - H_k\|\cdot \|d\|\\
             &\leq& \frac{L_2}{2}r_k^2\|\nabla f(x_k)\|^2+\frac{\lambda \|d\|^2 + L_VA_k\|d\|^2}{\lambda \|d\|^2 + \sigma_VA_k\|d\|^2 + \<H_kd,d>}\|\nabla f(x_k)\| \\
             &+& \|\nabla^2 f(x_k) - H_k\|\cdot r_k\|\nabla f(x_k)\|^{1/2}\\
             &\leq& \left(\xi + \frac{\lambda \|d\|^2 + L_VA_k\|d\|^2}{\lambda \|d\|^2 + \sigma_VA_k\|d\|^2 + \<H_kd,d>}\right)\|\nabla f(x_k)\|\\
             &\leq& (\xi + 1)\|\nabla f(x_k)\|.
        \end{eqnarray*}
        To finish the proof we have to notice, that if $\xi \leq \tfrac{\sqrt{5}-1}{2}$, then $\xi + 1 \leq 1/\xi$.
    \end{proof}

\begin{theorem}\textup{(Theorem \ref{th::convex})}
   Let Assumptions \ref{as:conv}, \ref{as:heslip}, \ref{as:sublevel}, \ref{as:rho} be satisfied, $H_k \succcurlyeq (L_V-\sigma_V)A_k\cdot I$ and $\xi \leq \tfrac{\sqrt{5}-1}{2}$. Then, to obtain $\e$-approximate zero-order stationary point, we need to perform 
    \[T \leq 2\sqrt{\frac{4D^3}{\e\eta^2}} + \log_{1/\xi}\frac{\|\nabla f(x_0)\|D}{\e}~~~\text{iterations}.\]
\end{theorem}
    \begin{proof}
        From Lemma \ref{lem:conv_g}, we get
        \[\|\nabla f(x_T)\| \leq \left(\frac{1}{\xi}\right)^{|\FF|}\xi^{|\GG|}\|\nabla f(x_0)\|.\]
        According to Lemma \ref{lem:conv_size}, $|\FF| \leq \sqrt{\frac{4D^3}{\e\eta^2}}$. Hence, $|\GG| \geq T - \sqrt{\frac{4D^3}{\e\eta^2}}$. Using the fact, that $\xi < 1$, we have
        \[\|\nabla f(x_T)\| \leq \left(\frac{1}{\xi}\right)^{\sqrt{\frac{4D^3}{\e\eta^2}}}\xi^{T - \sqrt{\frac{4D^3}{\e\eta^2}}}\|\nabla f(x_0)\| = \xi^{T - 2\sqrt{\frac{4D^3}{\e\eta^2}}}\|\nabla f(x_0)\|.\]

    From the convexity, we have 
    \[f(x_T) - f_* \leq \|\nabla f(x_T)\| \cdot D \leq  \xi^{T - 2\sqrt{\frac{4D^3}{\e\eta^2}}}\|\nabla f(x_0)\|\cdot D.\]
    Since $f(x_T) - f_* \geq \e$, then 
    \[T - 2\sqrt{\frac{4D^3}{\e\eta^2}} \leq \log_{\xi}\frac{\e}{\|\nabla f(x_0)\|D}\]
    and
    \[T \leq 2\sqrt{\frac{4D^3}{\e\eta^2}} + \log_{1/\xi}\frac{\|\nabla f(x_0)\|D}{\e}.\]
    \end{proof}    

\section{Valid Hessian Approximations}
In this section we examine some of the used objective functions in machine learning as well as Hessian approximations, that guarantee the bounded relative inexactness.
\begin{assumption}[star-convexity]
    The function $f : \RR^n \rightarrow \RR$ is star-convex, i.e.
    \[f(x_*) \geq f(x) + \<\nabla f(x),x_*-x>,~~~ \forall x\in\RR^n\]
\end{assumption}
\begin{assumption}[PL condition]
    The function $f : \RR^n \rightarrow \RR$ satisfy Polyak-Lojasiewicz (PL) condition, i.e.
    \[f(x)-f(x_*) \leq \frac{1}{2\mu}\|\nabla f(x)\|^2\]
\end{assumption}
\begin{lemma}
    Let the considered objective be an $l_2$ loss, i.e.
    \[f(x) = \frac{1}{2}\sum\limits_{i=1}^N \left(\phi(x,a_i)-b_i\right)^2,\]
    where $\phi(\cdot,a_i): \RR^n \rightarrow \RR$ has $L$-Lipschitz gradient.  Let the $H(x)$ be the General Gauss-Newton approximation:
    \[H(x) = \sum\limits_{i=1}^N\nabla_x \phi(x,a_i)\nabla_x \phi(x,a_i)^T.\]
    Then, we have 
    \begin{equation}
        \left\|\nabla^2 f(x) - H(x)\right\| \leq L\sqrt{2Nf(x)}.
    \end{equation}
    Additionally, if $f$ is star-convex, then,
    \[\left\|\nabla^2 f(x) - H(x)\right\| \leq L\sqrt{2N\|\nabla f(x)\|\cdot\|x-x_*\|} + L\sqrt{2Nf(x_*)}.\]
    In case $f$ satisfy PL condition, we have
    \[\left\|\nabla^2 f(x) - H(x)\right\| \leq L\|\nabla f(x)\|\sqrt{\frac{N}{\mu}} + L\sqrt{2Nf(x_*)}\]
\end{lemma}
\begin{proof}
    Take the second derivative of the objective, therefore, we have
    \begin{eqnarray*}
        \nabla^2 f(x) = \sum\limits_{i=1}^N \nabla_x \phi(x,a_i)\nabla_x \phi(x,a_i)^T + \sum\limits_{i=1}^N(\phi(x,a_i)-b_i)\nabla^2_x\phi(x,a_i).
    \end{eqnarray*}
    With simple algebra we obtain
    \begin{eqnarray*}
        \left\|\nabla^2 f(x) - H(x)\right\| &=& \left\|\sum\limits_{i=1}^N(\phi(x,a_i)-b_i)\nabla^2_x\phi(x,a_i)\right\|\\
        &\leq& \sum\limits_{i=1}^N|\phi(x,a_i)-b_i|\cdot\|\nabla^2_x\phi(x,a_i)\| \\
        &\leq& L \sum\limits_{i=1}^N|\phi(x,a_i)-b_i|\\
        &\leq& L \left(\sum\limits_{i=1}^N\left(\phi(x,a_i)-b_i\right)^2\right)^{1/2}\sqrt{N} = L\sqrt{2Nf(x)}.
    \end{eqnarray*}
    To derive the bounds for the star-convex functions, we can add and subtract the $f(x_*)$:
    \begin{eqnarray*}
        L\sqrt{2Nf(x)} &=& L\sqrt{2N(f(x)-f(x_*)) + 2Nf(x_*)} \leq L\sqrt{2N(f(x)-f(x_*))} + L\sqrt{2Nf(x_*)}\\
        &\leq& L\sqrt{2N\|\nabla f(x)\|\cdot\|x-x_*\|} + L\sqrt{2Nf(x_*)}.
    \end{eqnarray*}
    Under PL condition we have
    \begin{eqnarray*}
        L\sqrt{2N(f(x)-f(x_*))} + L\sqrt{2Nf(x_*)} \leq L\|\nabla f(x)\|\sqrt{\frac{N}{\mu}} + L\sqrt{2Nf(x_*)}
        \end{eqnarray*}
\end{proof}
\begin{corollary}
For overparametrized star-convex functions we have
\[\left\|\nabla^2 f(x) - H(x)\right\| \leq L\sqrt{2N\|\nabla f(x)\|\cdot\|x-x_*\|},\]
or
\[\left\|\nabla^2 f(x) - H(x)\right\| \leq  L\|\nabla f(x)\|\sqrt{\frac{N}{\mu}},\]
which is what we need for the relative inexactness.
\end{corollary}
\begin{lemma}
    Let the considered objective be a softmax loss, i.e.
    \[f(x) = -\sum\limits_{i=1}^N \log\frac{e^{\phi_{b_i}(x,a_i)}}{\sum_{j=1}^c e^{\phi_{b_j}(x,a_i)}},\]
    where $\phi(\cdot,a_i): \RR^n \rightarrow \RR^c$, and each of its component $\phi_i$ has $L$-Lipschitz gradient.  Let the $H(x)$ be the General Gauss Newton approximation:
    \[H(x) = -\sum\limits_{i=1}^N\left[\text{J}_x\phi(x,a_i)\right]^T\left[\frac{e^{\phi_{b_k}(x,a_k)}\cdot 1[k=l]}{\sum_{j=1}^c e^{\phi_{b_j}(x,a_i)}} - \frac{e^{\phi_{b_k}(x,a_k)}e^{\phi_{b_l}(x,a_l)}}{\left(\sum_{j=1}^c e^{\phi_{b_j}(x,a_i)}\right)^2}\right]_{k,l} \left[\text{J}_x\phi(x,a_i)\right]\]
    Then, we have 
    \begin{equation}
        \left\|\nabla^2 f(x) - H(x)\right\| \leq 2Lf(x).
    \end{equation}
    Additionally, if $f$ is star-convex, then,
    \[\left\|\nabla^2 f(x) - H(x)\right\| \leq 2L\|\nabla f(x)\|\cdot\|x-x_*\| + 2Lf(x_*).\]
    In case $f$ satisfy PL condition, we have
    \[\left\|\nabla^2 f(x) - H(x)\right\| \leq \frac{L}{\mu}\|\nabla f(x)\|^2 + 2Lf(x_*)\]
\end{lemma}
\begin{proof}
For the ease of notation we denote the 
$$p(b_i|\phi(x,a_i)) = \frac{e^{\phi_{b_i}(x,a_i)}}{\sum_{j=1}^c e^{\phi_{b_j}(x,a_i)}}.$$
Therefore, computing a Hessian, we get
\begin{eqnarray*}
    \nabla^2 f(x) &=& -\sum\limits_{i=1}^N\left[\text{J}_x\phi(x,a_i)\right]^T \nabla^2_\phi\left(-\log p(b_i|\phi(x,a_i)\right)\left[\text{J}_x\phi(x,a_i)\right] \\
    &&- \sum\limits_{i,j}\left[\nabla_f \log p(b_i|\phi(x,a_i))\right]_j \nabla_x^2 \phi_{j}(x,a_i)
\end{eqnarray*}
Therefore, 
\begin{eqnarray*}
-\nabla_{\phi} \log p(b_i|\phi(x,a_i)) &=&\nabla_\phi\left(-\phi_{b_i} + \left(\log\sum\limits_{j=1}^ce^{\phi_j(x,a_i)}\right)\right) \\
&=& \left(\frac{e^{\phi_1}}{\sum_j e^{\phi_j}},\frac{e^{\phi_2}}{\sum_j e^{\phi_j}}, \ldots,\frac{e^{\phi_i}}{\sum_j e^{\phi_j}} - 1 ,\ldots \right)^T 
\end{eqnarray*}
and
\[-\nabla^2_{\phi_l\phi_k}\nabla \log p(b_i|\phi(x,a_i)) = -\frac{e^{\phi_l}e^{\phi_k}}{\left(\sum_j e^{\phi_j}\right)^2}\]
\[-\nabla^2_{\phi_k^2}\nabla \log p(b_i|\phi(x,a_i)) = \frac{e^{\phi_k}}{\sum_j e^{\phi_j}} - \frac{e^{2\phi_k}}{\left(\sum_j e^{\phi_j}\right)^2}.\]
After all, we have 
\begin{eqnarray*}
    \left\|\nabla^2 f(x) - H(x)\right\|&=& \left\|\sum\limits_{i=1}^N\sum\limits_{j=1}^c \frac{\partial \log p(b_i|\phi(x,a_i))}{\partial \phi_j}\nabla^2\phi_j(x,a_i)\right\| \\
    &\leq& L\sum\limits_{i=1}^N\sum\limits_{j=1}^c\left|\frac{\partial \log p(b_i|\phi(x,a_i))}{\partial \phi_j}\right|\\
    &=& L\sum\limits_{i=1}^N\|\nabla_\phi \log p(b_i|x,a_i)\|_1 \\
    &=& L\sum\limits_{i=1}^N\left(\sum\limits_{j\neq b_i}\frac{e^{\phi_{j}(x,a_i)}}{\sum_k e^{\phi_k(x,a_i)}} + 1 - \frac{e^{\phi_{b_i}(x,a_i)}}{\sum_k e^{\phi_k(x,a_i)}}\right)\\
    &=& 2L\sum\limits_{i=1}^N\left(1 - \frac{e^{\phi_{b_i}(x,a_i)}}{\sum_k e^{\phi_k(x,a_i)}}\right)
\end{eqnarray*}
Since $1-x \leq -\log x$, we may derive
\begin{eqnarray*}
    \left\|\nabla^2 f(x) - H(x)\right\| \leq -2L\sum\limits_{i=1}^N\log\frac{e^{\phi_{b_i}(x,a_i)}}{\sum_k e^{\phi_k(x,a_i)}} = 2Lf(x).
\end{eqnarray*}
To derive the bounds for star-convex functions, we can add and subtract the $f(x_*)$:
\[2Lf(x) = 2L(f(x)-f(x_*)) + 2Lf(x_*) \leq 2L\|\nabla f(x)\|\cdot\|x-x_*\| + 2Lf(x_*).\]
Under PL condition we have
\[2L(f(x) - f(x_*)) + 2Lf(x_*) \leq \frac{L}{\mu}\|\nabla f(x)\|^2 + 2Lf(x_*)\]
\end{proof}
\begin{corollary}
For overparametrized star-convex functions we have
\[\left\|\nabla^2 f(x) - H(x)\right\| \leq 2L\|\nabla f(x)\|\cdot\|x-x_*\|,\]
or
\[\left\|\nabla^2 f(x) - H(x)\right\| \leq \frac{L}{\mu}\|\nabla f(x)\|^2,\]
which is what we need for the relative inexactness.
\end{corollary}
\textbf{Discussion and other examples.}

All the deterministic approaches, where the inexactness level is controllable, are valid to our approach, for instance, incorporating finite differences \cite{nocedal1999numerical, doikov2023first}. In these methods, we initialize the inexactness level $\e$ as $\|\nabla f(x)\|^{\beta}$. We believe, that stochastic approximations \cite{roosta2019sub, pilanci2017newton, agarwal2017second},  are also valid, at least in nonconvex scenario. However, corresponding convergence analysis needs to be derived at first.

\section{Additional numerical results}
\label{sec:add_numerical_results}

\paragraph{Additional relative inexactness experiments.}
We also present a new empirical result regarding relative inexactness $\Delta$.
We show that $\Delta$ remains bounded during training of a simple MLP model while using the SR-1 \cite{Conn1991ConvergenceOQ} and BFGS \cite{nocedal1999numerical} approximations.
Thus, we extend our results from Section \ref{sec:numerical_results} to other types preconditioning matrices.
This result is depicted in Figure \ref{fig:ri_extra}.

\begin{figure}[H]
    \subfloat{%
      \includegraphics[width=0.45\linewidth]{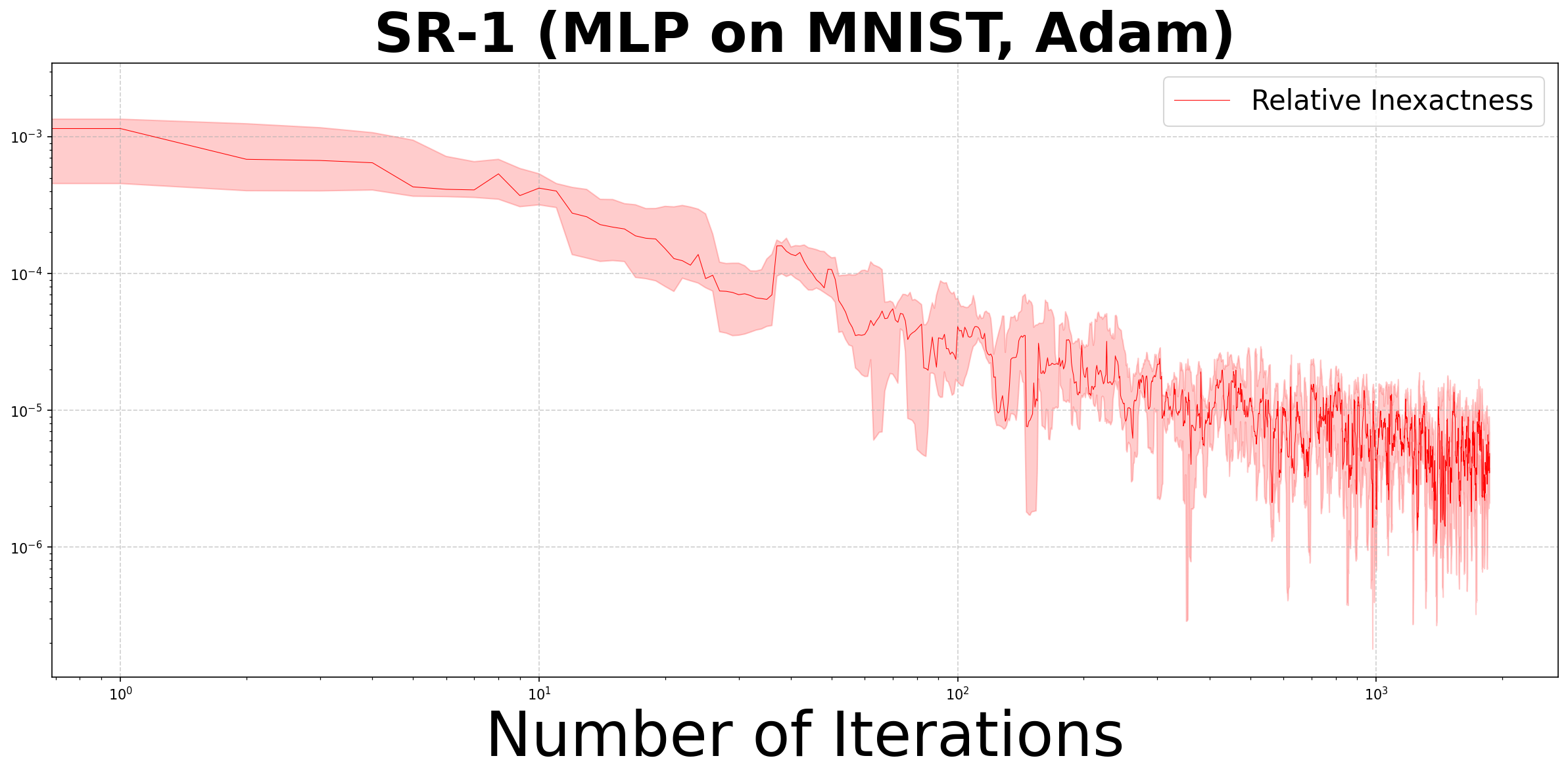}}%
    \hfill
    \subfloat{%
      \includegraphics[width=0.45\linewidth]{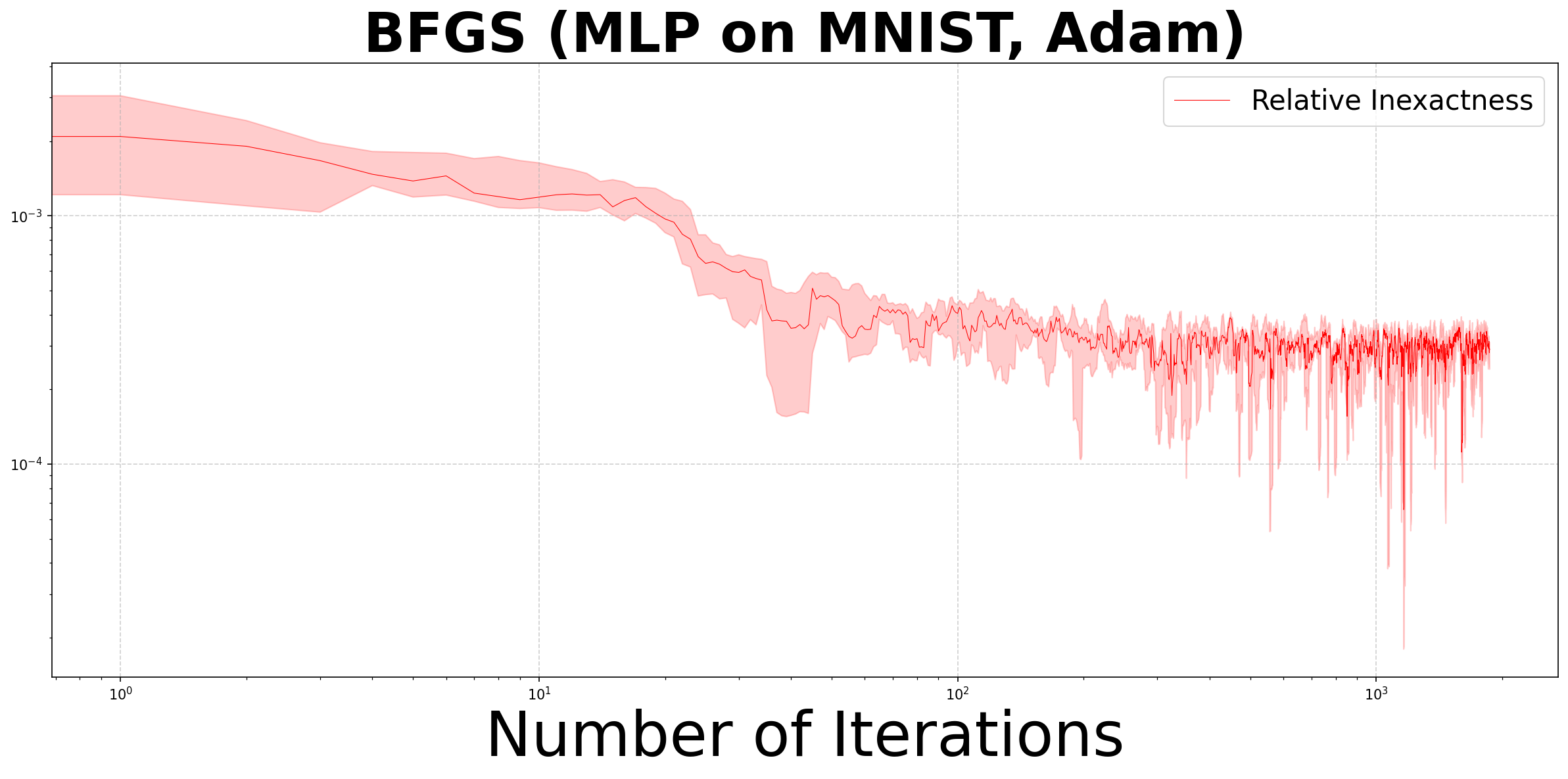}}%
  \caption{$\Delta$ evolution during Full training on MNIST.}\label{fig:ri_extra}
\end{figure}

\paragraph{HAT performance.}
We provide additional comparison between second-order, approximate second-order and first order methods on \texttt{mushrooms} dataset.
We included in this comparison: Adam \cite{kingma2014adam}, Broyden \cite{Broyden1965ACO}, DFP \cite{Davidon1959VariableMM,Fletcher1988PracticalMO}, BFGS \cite{nocedal1999numerical}, KFAC \cite{martens2015optimizing}. The results can be found in Figure \ref{fig:hat_mushrooms}

\begin{figure}[H]  
    \subfloat{%
      \includegraphics[width=0.45\linewidth]{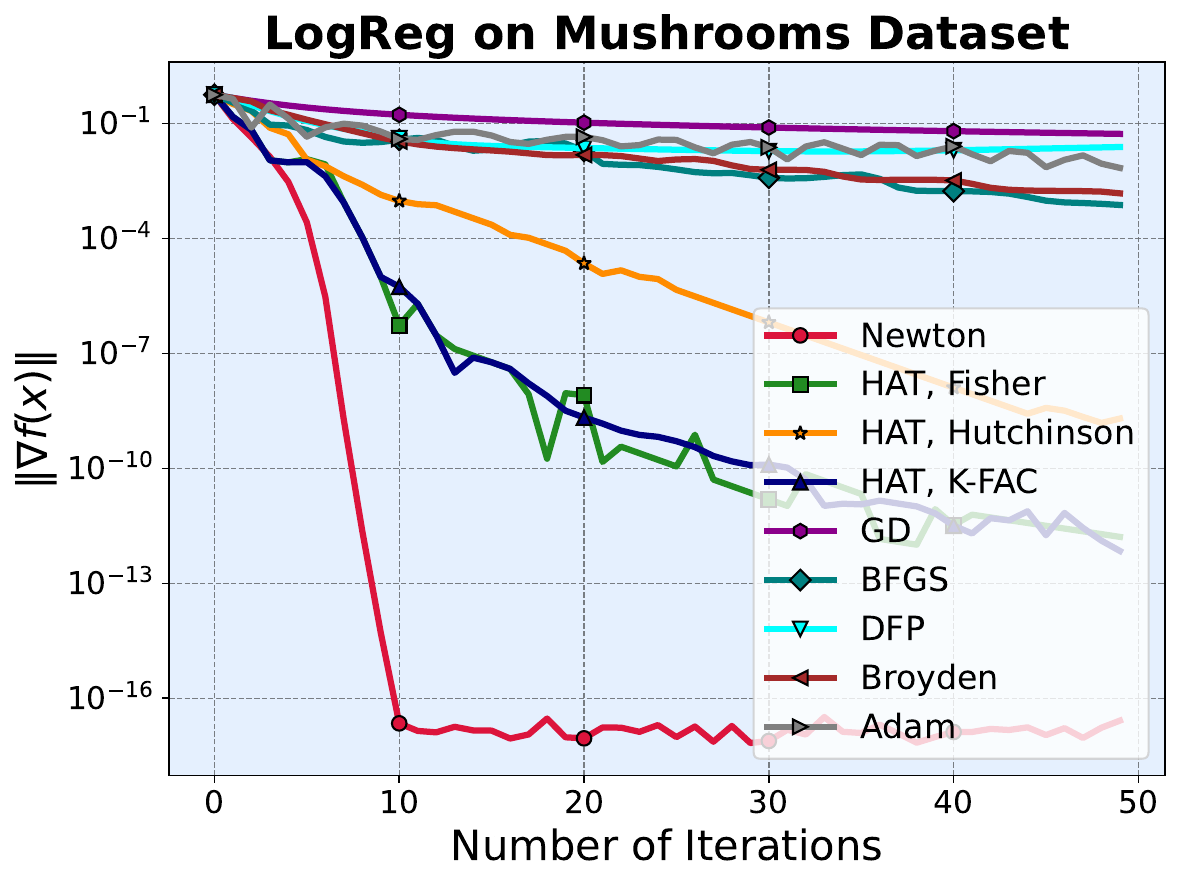}}%
    \hfill
    \subfloat{%
      \includegraphics[width=0.45\linewidth]{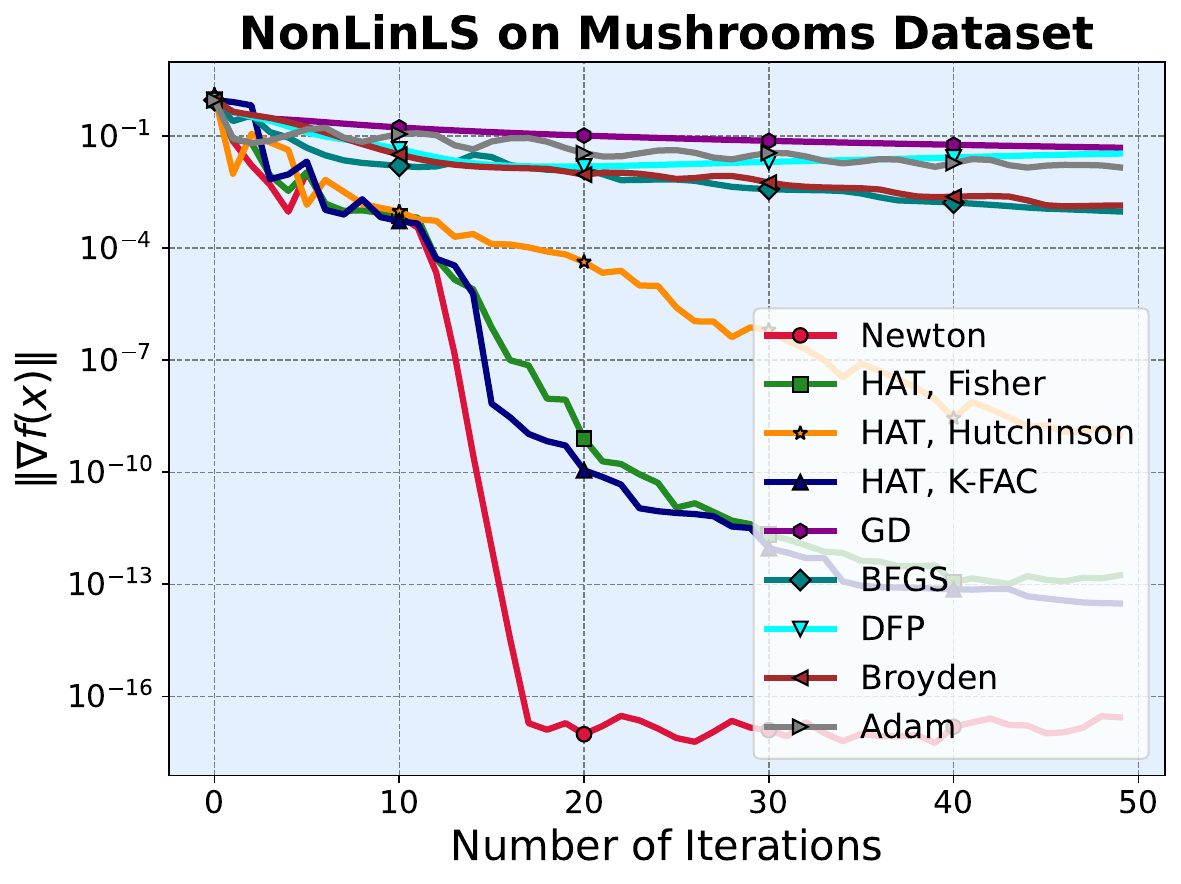}}%
\caption{Comparison of \textbf{HAT} with different methods on  \texttt{mushrooms} dataset}\label{fig:hat_mushrooms}
\end{figure}
Also, we show the convergence of HAT on highly nonconvex problem -- the Rosenbrock function $f(x_1, x_2) = (1 - x_1)^2 + 100(x_2 - x_1^2)^2$.
The result is in Figure \ref{fig:hat_nonconvex}. \textbf{HAT} obtains significantly better performance on such highly nonconvex problem, compared to other methods.
\begin{figure}[H]
    \centering
  \subfloat{%
    \includegraphics[width=0.45\linewidth]{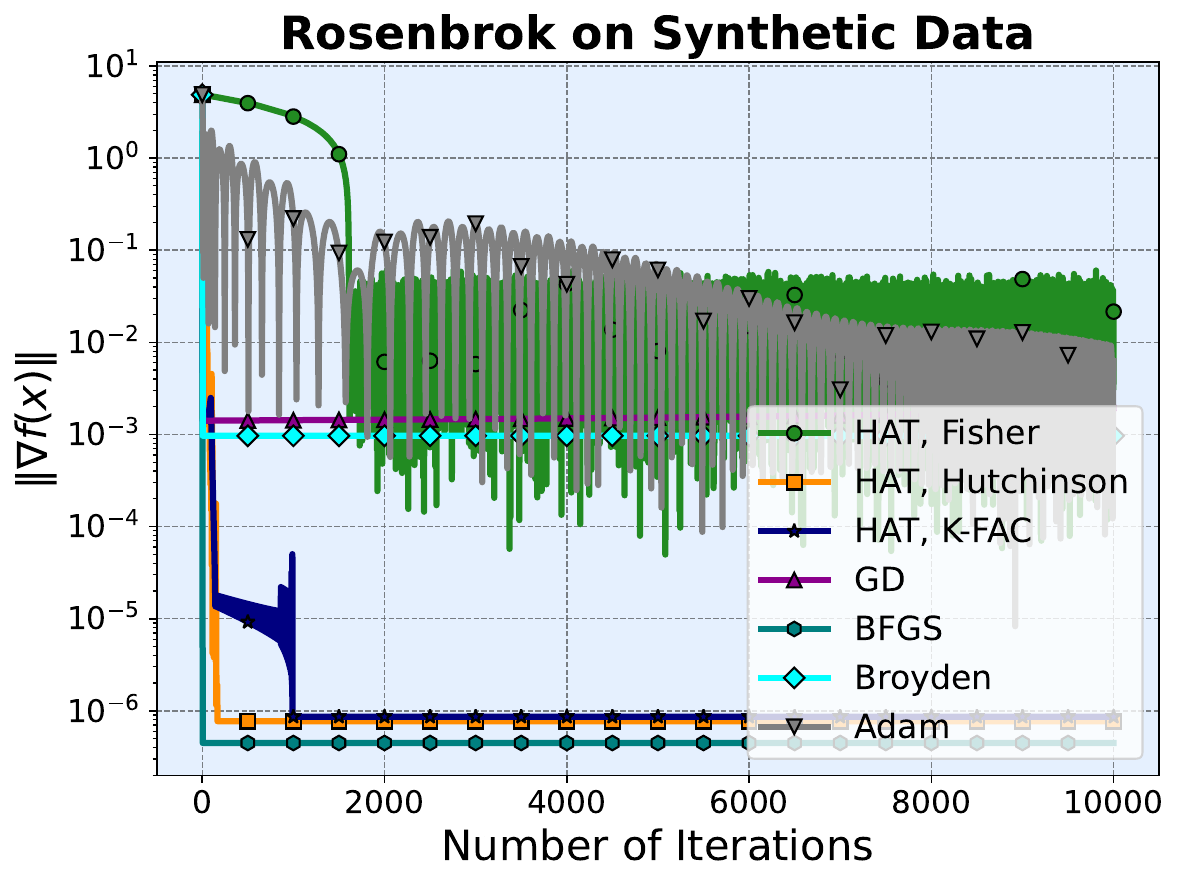}%
  }
  \caption{Comparison of \textbf{HAT} with different methods on Rosenbrock function}\label{fig:hat_nonconvex}
\end{figure}

\end{appendixpart}
\end{document}